\documentclass[11pt,reqno]{amsart}
\usepackage[utf8]{inputenc}
\usepackage[T1]{fontenc}
\usepackage[english]{babel}
\usepackage[usenames,dvipsnames]{xcolor}

\usepackage{hyperref}
\definecolor{darkred}{rgb}{0.7,0,0}
\definecolor{darkblue}{rgb}{0,0,0.7}
\hypersetup{colorlinks=true,urlcolor=darkred,linkcolor=darkred,citecolor=darkblue}
\usepackage{dsfont}
\usepackage{amsfonts}
\usepackage{amssymb}
\usepackage{amsmath}
\usepackage{amsthm}
\usepackage[left=3cm, right=3cm, top=2.2cm,bottom=2.2cm]{geometry}
\usepackage{mathtools}
\usepackage{etoolbox}

\usepackage{tikz}
\usepackage{tikz-cd}
\usetikzlibrary{arrows}
\tikzset{nodes={inner sep=0.2em}}
\tikzset{commutative diagrams/.cd,arrow style=tikz}

\DeclareRobustCommand\myiso{\xrightarrow{\!\smash{\raisebox{-0.25ex}{\ensuremath{\scriptstyle\sim}}\!}}}

\patchcmd{\section}{\scshape}{\bfseries}{}{}
\makeatletter
\renewcommand{\@secnumfont}{\bfseries}
\makeatother

\usepackage{enumitem}

\usepackage{cleveref}

\crefname{thm}{Theorem}{Theorems}
\crefname{prop}{Proposition}{Propositions}
\crefname{thmintro}{Theorem}{Theorems}
\crefname{propintro}{Proposition}{Propositions}
\crefname{lemma}{Lemma}{Lemmas}
\crefname{rem}{Remark}{Remarks}
\crefname{cor}{Corollary}{Corollaries}
\crefname{defi}{Definition}{Definitions}
\crefname{ex}{Example}{Examples}
\crefname{section}{Section}{Sections}

\newtheoremstyle{standard}{2ex}{2ex}{\itshape}{}{\bfseries}{.}{.5em}{}
\theoremstyle{standard}
\newtheorem{lemma}{Lemma}[section]
\newtheorem{prop}[lemma]{Proposition}
\newtheorem{cor}[lemma]{Corollary} 

\newtheorem{thmintro}{Theorem}[section]

\newtheoremstyle{definition}{2ex}{2ex}{}{}{\bfseries}{.}{.5em}{}    
\theoremstyle{definition}

\newtheorem{rem}[lemma]{Remark}

\newtheoremstyle{intermediate}{2ex}{2ex}{}{}{\itshape}{.}{.5em}{}    
\theoremstyle{intermediate}

\DeclareMathOperator{\Mod}{\mathsf{Mod}}
\DeclareMathOperator{\Alg}{\mathsf{Alg}}
\DeclareMathOperator{\Cat}{\mathsf{Cat}}
\DeclareMathOperator{\LFP}{\mathsf{LFP}}

\DeclareMathOperator{\cat}{\mathsf{cat}}
\DeclareMathOperator{\Set}{\mathsf{Set}}
\DeclareMathOperator{\Rex}{\mathsf{Rex}}
\DeclareMathOperator{\FinSet}{\mathsf{FinSet}}

\DeclareMathOperator{\SymPsMon}{\mathsf{SymPsMon}}

\DeclareMathOperator{\End}{End}
\DeclareMathOperator{\Hom}{Hom}
\DeclareMathOperator{\HOM}{\underline{\Hom}}
\DeclareMathOperator{\Ind}{Ind}
\DeclareMathOperator{\id}{id}

\DeclareMathOperator{\op}{op}
\DeclareMathOperator{\Ob}{Ob}
\DeclareMathOperator{\fp}{fp}

\renewcommand{\c}{\mathrm{c}}
\DeclareMathOperator{\fc}{fc}

\DeclareMathOperator{\coboxtimes}{\mathbin{\widehat{\boxtimes}}}

\newcommand{\A}{\mathcal{A}}
\newcommand{\B}{\mathcal{B}}
\newcommand{\C}{\mathcal{C}}
\newcommand{\D}{\mathcal{D}}
\newcommand{\E}{\mathcal{E}}

\newcommand{\I}{\mathcal{I}}

\newcommand{\J}{\mathcal{J}}

\renewcommand{\P}{\mathcal{P}}

\newcommand{\IK}{\mathds{K}}

\title[Bicategorical colimits of tensor categories]{Bicategorical colimits of tensor categories}
\author{Martin Brandenburg}
\thanks{\emph{E-mail address:} \texttt{brandenburg@uni-muenster.de}}

\begin{document}

\begin{abstract}
\noindent
In this expository paper we explain in detail how to construct bicategorical colimits of several kinds of tensor categories, for example essentially small finitely cocomplete $\IK$-linear tensor categories. The constructions are direct and elementary.
\end{abstract}

\maketitle

\section{Introduction}

In this expository paper we explain in detail how to construct bicategorical colimits of several kinds of tensor (i.e.\ symmetric monoidal) categories. In particular, we prove:

\begin{thmintro}\label{myThm}
The $2$-categories $\cat_{\otimes}$ and $\cat_{\fc\!\otimes}$ of essentially small (finitely cocomplete) tensor categories together with (finitely cocontinuous) tensor functors are bicategorically cocomplete. The same holds for the corresponding variants of $\IK$-linear structured categories $\cat_{\otimes/\IK}$ and $\cat_{\fc\!\otimes/\IK}$, where $\IK$ is any commutative ring. 
\end{thmintro}

Along the way, we also prove that the $2$-categories $\cat$ and $\cat_{\fc}$ of (finitely cocomplete) essentially small categories are bicategorically cocomplete.  It follows in particular:

\begin{thmintro}\label{myCor}
The $2$-category $\LFP_{\otimes}$, whose objects are locally finitely presentable tensor categories and whose morphisms are cocontinuous tensor functors preserving finitely presentable objects, is bicategorically cocomplete. Moreover, the inclusion $\LFP_{\otimes} \hookrightarrow \Cat_{\c\otimes}$ into the $2$-category of all cocomplete tensor categories preserves bicategorical colimits.
\end{thmintro}

This paper started as an appendix in the upcoming paper \cite{Bra20}, which continues the author's study of the interplay of tensor category theory and algebraic geometry via quasi-coherent sheaves \cite{Bra14}. But it seems reasonable to extract the purely category theoretic results and present them here in detail.

Essentially all results here are category theoretic folklore. In particular, \cref{myThm} can be obtained from $2$-dimensional monad theory \cite{BKP89}. For example, the constructions in \cite[Section 6]{BKP89} show that $\cat_{\fc\!\otimes}$ is the category of algebras of a finitary $2$-monad on the complete and cocomplete $2$-category $\cat$, so it is bicategorically cocomplete by \cite[Theorem 5.8]{BKP89}. But we will avoid this path and keep the proofs as direct and elementary as possible. For the same reason we avoid the usage of codescent objects \cite{Lac02}. Although they allow a construction of bicategorical pushouts of tensor categories without $2$-dimensional monad theory, checking all the details is cumbersome and requires further results from enriched category theory. Our proofs will perhaps be less elegant and do not follow the state of the art of category theory, but they are quite self-contained and hopefully easy to follow. We also review the necessary definitions and examples of bicategorical colimits in a preliminary section.

Some of our constructions of bicatetorical colimits of tensor categories may be seen as a categorificiation of the well-known constructions of colimits of commutative monoids or commutative rings. But one has to be more attentive to coherence data and coherence conditions. This also implies that the construction of bicategorical pushouts, for example, has to be done in two steps actually, the second one being not visible in the $1$-categorical setting, because it involves certain coequifiers. Schäppi also remarks in \cite[Section 5.6]{Sch15} that, for example, the $2$-category $\cat_{\fc\!\otimes/\IK}$ is bicategorically complete and cocomplete since it is the category of symmetric pseudomonoids in the bicategorically complete and cocomplete $2$-category $\cat_{\fc\!/\IK}$ and the usual arguments for existence of limits and colimits in the category of commutative algebras can be categorified. This paper explains in more detail how that categorification looks like.
 
I would like to thank Steve Lack, Michael Shulman, Alexander Campbell and Daniel Schäppi for their helpful comments on the topic.

\tableofcontents

\section{Preliminaries}
\label{sec:prelim}

\subsection{Tensor categories}

For us, a \emph{tensor category} is a symmetric monoidal category, and a \emph{tensor functor} is a strong symmetric monoidal functor \cite[Chapter XI]{ML98}. A \emph{(finitely) cocomplete tensor category} is a tensor category whose underlying category is (finitely) cocomplete in such a way that $\otimes$ preserves (finite) colimits in each variable. Such tensor categories may be seen as categorified rigs \cite{BD98,CJF13}. If $\IK$ is a commutative ring, a \emph{$\IK$-linear (cocomplete) tensor category} is a (cocomplete) tensor category whose underlying category is $\IK$-linear (i.e.\ enriched in $\Mod(\IK)$) such that the tensor product is a $\IK$-linear functor of both variables.

We also recall that a \emph{locally presentable tensor category} is a cocomplete tensor category whose underlying category is locally presentable in the sense of \cite{AR94,GU71}. They are automatically closed by Freyd's special adjoint functor theorem, so that the internal Hom-objects exist, which we denote by $\HOM$. A \emph{locally finitely presentable tensor category} is a cocomplete tensor category whose underlying category is locally finitely presentable and such that the finitely presentable objects are closed under finite tensor products \cite{Kel82}. In particular, it is required that the unit object (the empty tensor product) is finitely presentable. 

Since we are interested only in category theoretic notions which are invariant under equivalences, we abbreviate ``essentially small'' by ``small''. For instance, the category $\Mod_{\fp}(\IK)$ of finitely presentable $\IK$-modules is small for us.
 
We will denote by $\cat$, $\cat_{\fc}$, $\cat_{\otimes}$, $\cat_{\fc\!\otimes}$ the $2$-categories of small (finitely cocomplete) (tensor) categories and (finitely cocontinuous) (tensor) functors; see \cite{KR74} for an introduction to $2$-categories. A more common notation for $\cat_{\fc}$ is $\Rex$. The corresponding $2$-categories of $\IK$-linear structured categories will be denoted by $\cat_{\IK}$, $\cat_{\fc\!/ \IK}$, $\cat_{\otimes/ \IK}$, $\cat_{\fc\!\otimes/ \IK}$. These are the $2$-categories we will focus on. We will denote their Hom-categories in a similar way; for example $\Hom_{\fc\!\otimes}(\A,\B)$ denotes the category of finitely cocontinuous tensor functors from $\A$ to $\B$.

Occasionally we will also discuss the large variants: $\Cat_{\c}$ and $\Cat_{\c\otimes}$ denote the $2$-categories of cocomplete (tensor) categories. The (non-full) sub-$2$-categories of locally finitely presentable (tensor) categories together with cocontinuous (tensor) functors preserving finitely presentable objects are denoted by $\LFP$ and $\LFP_{\otimes}$. Again there are $\IK$-linear variants.
 
\subsection{Bicategorical colimits}
 
We refer to \cite{Ben67} for basic bicategorical concepts, and to \cite{Kel89,Str80,Str87} for the $2$-categorical and bicategorical (co-)limit notions which we recall in this section. Bicategorical limits and colimits are classically known as bilimits and bi\-co\-limits. However we avoid this terminology (following Schäppi \cite{Sch18} for example), mainly because of its overlap with already existing terms where ``bi'' refers to something which is two-sided (as in ``biproducts'').

The binary \emph{bicategorical coproduct} of two objects $A,B$ of a bicategory is an object $A \sqcup B$ with two morphisms $\iota_A : A \to A \sqcup B$, $\iota_B : B \to A \sqcup B$, inducing for all objects $T$ an equivalence of categories 
\[\Hom(A \sqcup B,T) \myiso \Hom(A,T) \times \Hom(B,T).\]
If this was even an isomorphism, then $A \sqcup B$ would be a binary \emph{$2$-categorical coproduct}. Generally, speaking, bicategorical universal properties talk about equivalences of categories, whereas $2$-categorical (i.e.\ $\cat$-enriched) universal properties talk about isomorphisms of categories. More generally, we can define the \emph{bicategorical coproduct} $\coprod_{i \in I} A_i$ of an arbitrary family of objects $(A_i)_{i \in I}$. The case $I=\emptyset$ corresponds to a \emph{bicategorical initial object}.

Given a directed preorder $(I,\leq)$, objects $A_i$ for $i \in I$, morphisms $f_{i,j} : A_i \to A_j$ for $i \leq j$ and isomorphisms $\id_{A_i} \myiso f_{i,i}$, $f_{j,k} \circ f_{i,j} \myiso f_{i,k}$ satisfying three evident coherence conditions, a \emph{bicategorical directed colimit} of these data is an object $A$ with morphisms $u_i : A_i \to A$ and isomorphisms $u_j \circ f_{i,j} \myiso f_i$ which are compatible with the given isomorphisms and such that the obvious bicategorical universal property is satisfied. Arbitrary bicategorical coproducts can be written as bicategorical directed colimits of finite bicategorical coproducts, assuming these exist.
 
If two morphisms $f : C \to A$, $g : C \to B$ are given, a \emph{bicategorical pushout} $A \sqcup_C B$ has two morphisms $\iota_A : A \to A \sqcup_C B$, $\iota_B : B \to A \sqcup_C B$ as well as an isomorphism $\alpha : \iota_A \circ f \myiso \iota_B \circ g$ such that for all objects $T$ we obtain an equivalence of categories
\[\Hom(A \sqcup_C B,T) \myiso \Hom(A,T) \times_{\Hom(C,T)} \Hom(B,T).\]
Here, the right hand side is a \emph{bicategorical pullback} of categories, so that its objects are triples $(u : A \to T,\, v : B \to T,\, \beta : u \circ f \myiso v \circ g)$.

If $f,g : A \rightrightarrows B$ is a parallel pair of morphisms in a bicategory, a \emph{bicategorical coequalizer} is a morphism $p : B \to C$ equipped with an isomorphism $\alpha : p \circ f \myiso p \circ g$ such that the corresponding bicategorical universal property is satisfied: For every object $T$ it is required that $(p,\alpha)$ induces an equivalence of categories between $\Hom(C,T)$ and the category of pairs $(q,\beta)$, where $q \in \Hom(B,T)$ and $\beta : q \circ f \myiso q \circ g$. 

If more generally $(f_i,g_i : A \rightrightarrows B)_{i \in I}$ is a family of parallel pairs, a \emph{multiple bicategorical coequalizer} is a bicategorically universal morphism $p : B \to C$ equipped with a family of isomorphisms $(\alpha_i : p \circ f_i \myiso p \circ g_i)_{i \in I}$. If the coproduct $\coprod_{i \in I} A$ exists, the multiple bicategorical coequalizer can be simply seen as the (single) bicategorical coequalizer of the parallel pair $(f_i)_{i \in I},(g_i)_{i \in I} : \coprod_{i \in I} A \rightrightarrows B$. Alternatively, if $I$ is finite, it may be simply realized as a composition of single bicategorical coequalizers.
 
A \emph{bicategorical coinserter} of a parallel pair $f,g : A \rightrightarrows B$ is defined almost like a bicategorical coequalizer $(p,\alpha)$, except that $\alpha : p \circ f \to p \circ g$ is merely a morphism. For this reason bicategorical coequalizers are also known as bicategorical coisoinserters. Again there is a notion of a multiple bicategorical coinserter, which can be reduced to single bicategorical coinserters if the index set is finite or if bicategorical coproducts exist. In case there is a morphism $p \circ f \to p \circ g$, we say that $p$ \emph{coinserts} $f$ and $g$ (via the morphism).
 
We know from $1$-dimensional category theory that, assuming that binary coproducts exist, one can construct pushouts from coequalizers and vice versa. The same is true in the bicategorical setting. We omit the easy proofs.
   
\begin{lemma} \label{pushviacoeq}
If a bicategory has binary bicategorical coproducts and bicategorical coequalizers, then it has bicategorical pushouts. Specifically, the bicategorical pushout of two morphisms $f : C \to A$, $g : C \to B$ is given by the bicategorical coequalizer of $\iota_A \circ f : C \to A \sqcup B$ and $\iota_B \circ g : C \to A \sqcup B$. \hfill $\square$
\end{lemma}

\begin{lemma} \label{coeqviapush}
If a bicategory has binary bicategorical coproducts and bicategorical pushouts, then it has bicategorical coequalizers. Specifically, the bicategorical coequalizer of a parallel pair $f,g : A \rightrightarrows B$ is the bicategorical pushout of $(f,g) : A \sqcup A \to B$ and the codiagonal morphism $\nabla : A \sqcup A \to A$. \hfill $\square$
\end{lemma}

The idea of a coequalizer is to make two parallel morphisms equal. The corresponding construction for $2$-morphisms is known as a coequifier: Given morphisms $f,g : A \rightrightarrows B$ and morphisms $\alpha,\beta : f \rightrightarrows g$, a \emph{bicategorical coequifier} of $\alpha,\beta$ is a morphism $p : B \to C$ with $p \circ \alpha = p \circ \beta$ (both sides are morphisms $p \circ f \to p \circ g$), such that the evident bicategorical universal property holds: For every object $T$ we get an equivalence between $\Hom(C,T)$ and the category of those $q : B \to T$ satisfying $q \circ \alpha = q \circ \beta$, in which case we say that $q$ \emph{coequifies} $\alpha,\beta$. Again there is a notion of a multiple bicategorical coequifier, which can be reduced to single bicategorical coequifiers if the index set is finite or if bicategorical coproducts exist.

Given two morphisms $f,g : A \rightrightarrows B$ and a morphism $\alpha : f \to g$, a \emph{bicategorical coinverter} of $\alpha$ is a bicategorically universal morphism $p : B \to C$ such that $p \circ \alpha : p \circ f \to p \circ g$ is an isomorphism. In $\cat$ bicategorical coinverters are known as \emph{localizations}.

\begin{lemma} \label{coeqviains}
Assume that a bicategory has bicategorical coequifiers and bicategorical coinserters. Then it has bicategorical coinverters and bicategorical coequalizers.
\end{lemma}

\begin{proof}
To construct the bicategorical coinverter of $\alpha : f \to g  : A \rightrightarrows B$, we start with the bicategorical coinserter $(p : B \to C,\, \beta : p \circ g \to p \circ f)$ of $g,f$ and then consider the multiple bicategorical coequifier $q : C \to D$ of $((p \circ \alpha) \circ \beta, \id_{p \circ g})$ and $(\beta \circ (p \circ \alpha),\id_{p \circ f})$. Then $q \circ p : B \to D$ is the bicategorical coinverter of $\alpha$.

To construct the bicategorical coequalizer of $f,g : A \rightrightarrows B$, we first form the bicategorical coinserter $(p : B \to C , \alpha : p \circ f \to p \circ  g)$ of $f,g$ and then the bicategorical coinverter of $\alpha$.
\end{proof}

\begin{rem} \label{pushviacoprodcoeqcoins}
\cref{pushviacoeq} and \cref{coeqviains} imply that a bicategory with binary bicategorical coproducts, bicategorical coinserters and bicategorical coequifiers has bicategorical pushouts. For our application in \cite{Bra20} only the latter are relevant, so the readers coming from there might as well skip the rest of this section, which is about more general types of bicategorical colimits.
\end{rem}

If $X$ is a small category and $A$ is an object of a bicategory $\C$, then the \emph{bicategorical tensor product} $X *_b A$ is an object of $\C$ with a functor $X \to \Hom_\C(A,X *_b A)$, inducing for every other object $T \in \C$ an equivalence of categories
\[\Hom_\C(X *_b A, T) \myiso \Hom\bigl(X,\Hom_\C(A,T)\bigr).\]

\begin{lemma} \label{bicattens}
If a bicategory has bicategorical coproducts, bicategorical coinserters and bicategorical coequifiers, then it has bicategorical tensor products.
\end{lemma}

\begin{proof}
Kelly has proven the corresponding statement for $2$-categorical limits (and hence also colimits) in \cite[Proposition 4.4]{Kel89}, but the proof can easily be adapted.
\end{proof}
 
The notion of a bicategorical tensor product can be generalized to a \emph{bicategorical weighted colimit}, which in fact encapsulates all mentioned examples, but conversely can also be built from them. Let $\C$ be a bicategory. Let $\J$ be a small bicategory and $X : \J^{\op} \to \cat$ be a homomorphism of bicategories, called a \emph{weight}. If $A : \J \to \C$ is a homomorphism of bicategories, then $X *_b A$ is an object of $\C$ with natural equivalences of categories
\[\Hom_\C(X *_b A, T) \myiso \Hom\bigl(X,\Hom_{\C}(A(-),T)\bigr).\]
If $X *_b A$ always exists, we say that $\C$ is \emph{bicategorically cocomplete}.

\begin{prop} \label{allcolim}
If a bicategory has bicategorical coproducts, bicategorical coinserters and bicategorical coequifiers, then it is bicategorical cocomplete.
\end{prop}

\begin{proof}
By \cref{bicattens} the bicategory has bicategorical tensor products, and by \cref{coeqviains} it has bicategorical coequalizers. Street has shown in \cite[Section 1]{Str80}, with a correction in \cite{Str87}, that a bicategory with bicategorical coproducts, bicategorical coequalizers and bicategorical tensor products is bicategorically cocomplete; see also \cite[Section 6]{Kel89}.
\end{proof}
 
For this reason we will concentrate on these three types of bicategorical colimits in the rest of the paper.

\section{Bicategorical coproducts}

In this section we prove that the four $2$-categories of tensor categories mentioned before have arbitrary bicategorical coproducts.
  
\begin{prop} \label{coprodtensor}
The $2$-categories $\cat_{\otimes}$, $\cat_{\fc\!\otimes}$, $\cat_{\otimes/\IK}$, $\cat_{\fc\!\otimes/\IK}$ have binary bicategorical coproducts.
\end{prop}

\begin{proof}
We start with $\cat_{\otimes}$. Let $\A,\B$ be two small tensor categories. We denote their unit objects by $1_\A$ and $1_\B$. The bicategorical coproduct of $\A$ and $\B$ is just the product tensor category $\A \times \B$ (with componentwise structure) with the tensor functors $\iota_\A : \A \to \A \times \B$, $\iota_\A(A) \coloneqq (A,1_\B)$ and $\iota_\B : \B \to \A \times \B$, $\iota_\B(B) \coloneqq (1_\A,B)$. The tensor structure on $\iota_\A$ is given by the identity $\iota_\A(1_\A)=(1_\A,1_\B)$ and the natural isomorphisms
\[\iota_\A(A) \otimes \iota_\A(A') = (A \otimes A',1_\B \otimes 1_\B)  \myiso (A \otimes A',1_\B) = \iota_\A(A \otimes A').\]
The tensor structure on $\iota_\B$ is defined in the same way. If $\C$ is a tensor category and $F : \A \to \C$, $G : \B \to \C$ are two tensor functors, then they induce a tensor functor $H : \A \times \B \to \C$ by $H(A,B) \coloneqq F(A) \otimes G(B)$. The tensor structure on $H$ is given by
\[1_\C \myiso 1_\C \otimes 1_\C \myiso F(1_\A) \otimes G(1_\B) = H(1_\A,1_\B)\]
and the natural isomorphisms (using in particular the symmetry in $\C$)
\begin{align*}
H(A,B) \otimes H(A',B') & = \bigl(F(A) \otimes G(B)\bigr) \otimes\bigl (F(A') \otimes G(B')\bigr) \\
& \myiso \bigl(F(A) \otimes F(A')\bigr) \otimes \bigl(G(B) \otimes G(B')\bigr) \\
&\myiso F(A \otimes A') \otimes G(B \otimes B') \\
&= H(A \otimes A',B \otimes B').
\end{align*}
One checks that this describes an equivalence of categories
\[\Hom_{\otimes}(\A \times \B,\C) \simeq \Hom_{\otimes}(\A,\C) \times \Hom_{\otimes}(\B,\C).\]
This is just a categorification of the well-known construction of the coproduct of two commutative monoids.

For $\cat_{\otimes/ \IK}$ almost the same construction works. We just replace $\A \times \B$ by the tensor product $\A \otimes_{\IK} \B$ which is defined by $\Ob(\A \otimes_{\IK} \B) \coloneqq \Ob(\A) \times \Ob(\B)$ and
\[\Hom_{\A \otimes_{\IK} \B} ((A,B),(A',B')) \coloneqq \Hom_\A(A,A') \otimes_{\IK} \Hom_\B(B,B').\]
In fact, notice that for two $\IK$-linear tensor functors $F : \A \to \C$, $G : \B \to \C$ the tensor functor $H : \A \times \B \to \C$ constructed above is $\IK$-linear in each variable, thus corresponds to a $\IK$-linear tensor functor $\smash{\tilde{H} : \A \otimes_{\IK} \B \to \C}$.
 
In order to construct coproducts in $\cat_{\fc\!\otimes}$, we use the tensor product $\A \boxtimes \B$ of small finitely cocomplete categories $\A,\B$ (cf.\ \cite[Section 6.5]{Kel05}) which is defined by the universal property
\[\Hom_{\fc}(\A \boxtimes \B,\C) \simeq \{F \in \Hom(\A \times \B,\C) : F \text{ is finitely cocontinuous in each variable}\}.\]
So it comes equipped with a functor $\A \times \B \to \A \boxtimes \B$, denoted $(A,B) \mapsto A \boxtimes B$, which is finitely cocontinuous in each variable and bicategorically universal with this property; the objects in the image are called \emph{elementary tensors}. In order to construct it, we start with the product category $\P \coloneqq \A \times \B$ and consider its free cocompletion $\smash{\widehat{\P} \coloneqq \Hom(\P^{\op},\Set)}$ with the Yoneda embedding $\smash{Y : \P \hookrightarrow \widehat{\P}}$. The problem is that $Y$ does not preserve finite colimits in both variables. In order to fix thus, consider the set $\Phi$ of cocones in $\A \times \B$ which are either of the form $((A_i,B) \to (A,B))$ for finite colimit cocones $(A_i \to A)$ in $\A$ and objects $B \in \B$ or of the form $((A,B_i) \to (A,B))$ for finite colimit cocones $(B_i \to B)$ in $\B$ and objects $A \in \A$. To be precise, there is only a set of such cocones up to isomorphism, but this is of course sufficient.

Let $\smash{\Alg(\Phi) \subseteq \widehat{\P}}$ denote the full subcategory of presheaves which map all cocones in $\Phi$ to limit cones in $\Set$. This is actually a reflective subcategory by \cite[Theorem 6.11]{Kel05} (alternatively, one may use \cite[Theorem 1.39]{AR94}), in particular cocomplete. Let $\smash{R : \widehat{\P} \to \Alg(\Phi)}$ be the reflector. Now let $\A \boxtimes \B \subseteq \Alg(\Phi)$ denote the smallest full subcategory of $\Alg(\Phi)$ which is closed under finite colimits and contains the image of $\smash{R \circ Y : \P \to \widehat{\P} \to \Alg(\Phi)}$. Thus, $\A \boxtimes \B$ is finitely cocomplete. If $\C$ is any finitely cocomplete category, then by \cite[Theorem 6.23]{Kel05} the category $\Hom_{\fc}(\A \boxtimes \B,\C)$ is equivalent to the category of $\Phi$-comodels in $\C$, which by definition are functors $F : \P \to \C$ which map the cocones in $\Phi$ to colimit cocones in $\C$, i.e.\ functors $F : \A \times \B \to \C$ which are finitely cocontinuous in each variable.

Now if $\A,\B$ are finitely cocomplete tensor categories, then $\A \boxtimes \B$ becomes a finitely cocomplete tensor category with unit object $1_\A \boxtimes 1_\B$ and tensor product
\[\begin{tikzcd}[sep=6ex]
(\A \boxtimes \B) \boxtimes (\A \boxtimes \B) \ar{r}{\sim} & (\A \boxtimes \A) \boxtimes (\B \boxtimes \B) \ar{rr}{{\otimes_\A}\, \boxtimes\,\, {\otimes_\B}} && \A \boxtimes \B.
\end{tikzcd}\]
It satisfies the universal property of a coproduct in $\cat_{\fc\!\otimes}$: Given two finitely cocontinuous tensor functors $F : \A \to \C$, $G : \B \to \C$, the functor $H : \A \times \B \to \C$ defined above is finitely cocontinuous in each variable, thus extends to a finitely cocontinuous tensor functor $\smash{\tilde{H} : \A \boxtimes \B \to \C}$, which can be given a tensor structure as well, using the tensor structures on $F$ and $G$.
  
For $\cat_{\fc\!\otimes/ \IK}$ we use a corresponding tensor product $\A \boxtimes_{\IK} \B$ which classifies functors which are finitely cocontinuous and $\IK$-linear in each variable (see also \cite[Theorem 7]{LF13}).
\end{proof}

\begin{rem} \label{unify1}
It is possible to unify all four cases in \cref{coprodtensor} using a general result by Schäppi \cite[Theorem 5.2]{Sch14}, since $\cat_{\otimes}$, $\cat_{\otimes/\IK}$, $\cat_{\fc\!\otimes}$, $\cat_{\fc\!\otimes/\IK}$ are the categories of symmetric pseudomonoids in the symmetric monoidal $2$-categories $(\cat,\times)$, $(\cat_{\otimes/\IK},\otimes_K)$, $(\cat_{\fc},\boxtimes)$, $(\cat_{\fc\!/\IK},\boxtimes_{\IK})$; the special case of $\cat_{\fc\!\otimes/\IK}$ is \cite[Theorem 5.1]{Sch14}. Schäppi's general result is more advanced, though. For our four types of tensor categories this level of generality is not strictly necessary.
\end{rem}

\begin{rem} \label{smallred}
Let $\A$ be a small finitely cocomplete tensor category. Let $\C$ be any cocomplete tensor category. If $\C'$ runs through all small full subcategories of $\C$ which are closed under finite colimits and finite tensor products, then the canonical functor
\[{\varinjlim}_{\C'} \Hom_{\fc\!\otimes}(\A,\C') \to \Hom_{\fc\!\otimes}(\A,\C)\]
is clearly an equivalence of categories. From this we deduce that the bicategorical coproduct $\A \boxtimes \B$ of two small finitely cocomplete tensor categories actually satisfies
\[\Hom_{\fc\!\otimes}(\A \boxtimes \B,\C) \simeq \Hom_{\fc\!\otimes}(\A,\C) \times \Hom_{\fc\!\otimes}(\B,\C)\]
for every cocomplete tensor category $\C$, even though $\C$ does not have to be small. Incidentally this also results from the construction and the proof of the universal property of $\A \boxtimes \B$. The corresponding statements hold in the $\IK$-linear case.
\end{rem}

\begin{prop} \label{coprodlfp}
If $\A,\B$ are two locally finitely presentable tensor categories, then they have a bicategorical coproduct $\smash{\A \coboxtimes \B}$ inside the $2$-category of all cocomplete tensor categories and cocontinuous tensor functors. Moreover, $\smash{\A \coboxtimes \B}$ is a locally finitely presentable tensor category, and the inclusion functors $\smash{\A \rightarrow \A \coboxtimes \B \leftarrow \B}$ preserve finitely presentable objects. The corresponding statements hold in the $\IK$-linear case.
\end{prop}

\begin{proof}
The category $\A_{\fp}$ of finitely presentable objects in $\A$ is a small finitely cocomplete tensor category, and the same holds for $\B_{\fp}$. From \cref{coprodtensor} we know that there is a bicategorical coproduct $\A_{\fp} \boxtimes \B_{\fp}$ in $\cat_{\fc\!\otimes}$. We define
\[\A \coboxtimes \B \coloneqq \Ind(\A_{\fp} \boxtimes \B_{\fp}),\]
see \cite[Chapter 6]{KS06} for Indization. This is a locally finitely presentable tensor category with $\smash{(\A \coboxtimes \B)_{\fp} = \A_{\fp} \boxtimes \B_{\fp}}$. The finitely cocontinuous tensor functor $\A_{\fp} \to \A_{\fp} \boxtimes \B_{\fp}$ extends to a cocontinuous tensor functor $\smash{\A = \Ind(\A_{\fp}) \to \Ind(\A_{\fp} \boxtimes \B_{\fp}) = \A \coboxtimes \B}$. It preserves finitely presentable objects by construction. The same holds for $\smash{\B \to \A \coboxtimes \B}$. If $\C$ is any cocomplete tensor category, then
\begin{align*}
\smash{\Hom_{\c\otimes}(\A \coboxtimes \B,\C)} & = \Hom_{\c\otimes}\bigl(\Ind(\A_{\fp} \boxtimes \B_{\fp}),\C\bigr) \\
& \simeq \Hom_{\fc\!\otimes}(\A_{\fp} \boxtimes \B_{\fp},\C) \\
& \simeq \Hom_{\fc\!\otimes}(\A_{\fp},\C) \times \Hom_{\fc\!\otimes}(\B_{\fp},\C) \\
& \simeq \Hom_{\c\otimes}(\A,\C) \times \Hom_{\c\otimes}(\B,\C),
\end{align*}
where we have used \cref{smallred} in the second equivalence.
\end{proof}

So far we have only constructed binary bicategorical coproducts. Let us generalize this:
 
\begin{prop} \label{allcoprodtensor}
The $2$-categories $\cat_{\otimes}$, $\cat_{\otimes/\IK}$, $\cat_{\fc\!\otimes}$, $\cat_{\fc\!\otimes/\IK}$ have arbitrary bicategorical coproducts.
\end{prop}

\begin{proof}
The $2$-categories have bicategorical initial objects, namely (in this order) $\{1\}$ with $\End(1)=\{\id_1\}$, $\{1\}$ with $\End(1)=\IK$, $\FinSet$ with $\otimes=\times$, $\Mod_{\fp}(\IK)$ with $\otimes=\otimes_{\IK}$. We have seen in \cref{coprodtensor} that binary bicategorical coproducts exist. Thus, finite bicategorical coproducts exist. So what is left to prove is that bicategorical directed colimits exist. We only sketch the proof.

Let $(I,\leq)$ be a directed preorder. We fix some index $i_0 \in I$ and a function $u : I \times I \to I$ with $i \leq u(i,j)$ and $j \leq u(i,j)$. Let $(\A_i)_{i \in I}$ be a family of small categories equipped with functors $F_{i,j} : \A_i \to \A_j$ for $i \leq j$ and isomorphisms $\id_{\A_i} \myiso F_{i,i}$, $F_{j,k} \circ F_{i,j} \myiso F_{i,k}$ which satisfy the three obvious coherence conditions. Their bicategorical colimit $\A$ in $\cat$ has as objects the $(i,A)$ with $i \in I$ and $A \in \A_i$, and the morphisms are
\[\Hom((i,A),(j,B)) \coloneqq {\varinjlim}_{k \geq i,j} \Hom\bigl(F_{i,k}(A),F_{j,k}(B)\bigr).\]
If we start with a diagram in $\cat_{\otimes}$ as above, then $\A$ becomes a tensor category: The unit object is $({i_0},1_{\A_{i_0}})$, and the tensor product is defined on objects by
\[(i,A) \otimes (j,B) \coloneqq \bigl(u(i,j),F_{i,u(i,j)}(A) \otimes F_{j,u(i,j)}(B)\bigr).\]
In fact, then $\A$ is a bicategorical directed colimit in $\cat_{\otimes}$. The conceptual reason for this is that $\times : \cat^2 \to \cat$ preserves bicategorical directed colimits in each variable. For $\cat_{\fc}$ we have to do the same reflection trick as in the proof of \cref{coprodtensor} to find a finitely cocomplete category $\smash{\B \subseteq \widehat{\A}}$ with finitely cocontinuous functors $\A_n \to \B$. Again the bicategorical directed colimit in $\cat_{\fc}$ can be used for $\cat_{\fc\!/\otimes}$ as well. The construction for the $\IK$-linear variants is similar.
\end{proof}

\begin{rem} \label{coprodcat}
One can also show that $\cat$, $\cat_{\IK}$, $\cat_{\fc}$, $\cat_{\fc\!/\IK}$ have arbitrary bicategorical coproducts. Since we do not need this for tensor categories in the following, we only sketch the construction. For $\cat$ the usual disjoint union is even a $2$-categorical coproduct. For $\cat_{\IK}$ we tweak this disjoint union by defining $\Hom(A,B) \coloneqq \{0\}$ for objects $A,B$ of different categories. For $\cat_{\fc}$ a categorified version of ``restricted products'' works, namely the subcategory of the product where almost all objects are initial. For $\cat_{\fc\!/\IK}$ we do the reflection trick again.
\end{rem}
 
\section{Bicategorical coinserters and coequifiers}

In this section we prove that the $2$-categories of tensor categories mentioned before have bicategorical coinserters and coequifiers. Combining this with the previous sections, this implies the existence of all bicategorical colimits. We need to look at categories without tensor structure first.
 
\begin{prop} \label{coins}
The $2$-categories $\cat$, $\cat_{\IK}$, $\cat_{\fc}$, $\cat_{\fc\!/\IK}$ have bicategorial coinserters.
\end{prop}

\begin{proof}
We start with $\cat$. Let $F,G : \A \rightrightarrows \B$ be two functors between small categories. We may assume that $\B$ is small in the usual sense, i.e.\ that $\Ob(\B)$ is a set. To the underlying digraph of $\B$ we add edges of the form $\alpha_A : F(A) \to G(A)$ for each $A \in \A$.  The edge associated to a morphism $f : B \to B'$ in $\B$ is denoted by $[f]$. On the path category generated by this digraph we consider the congruence relation $R$ generated by the following two sets of relations: (1) the relations in $\B$, by which we mean that $[\id_B]$ is equivalent to $\id_B$ for $B \in \B$ and that $[g \circ f]$ is equivalent to $[g] \circ [f]$ whenever $f,g$ are two morphisms in $\B$ for which $g \circ f$ is defined, (2) for every morphism $f : A \to A'$ in $\A$ the composition $[G(f)] \circ \alpha_A$ is equivalent to $\alpha_{A'} \circ [F(f)]$. The quotient category $\C \coloneqq \B/R$ is a small category with a functor $P : \B \to \C$ and a morphism $\alpha : P \circ F \to P \circ G$. This is clearly the $2$-categorical and hence bicategorical coinserter of $F,G$.

For $\cat_{\IK}$ a similar construction works: We just consider the free $\IK$-linear category on the path category from above and then take the quotient with respect to the relations from above together with the relations which ensure that $P$ becomes $\IK$-linear.

Now consider two maps $F,G : \A \rightrightarrows \B$ in $\cat_{\fc}$. We apply the forgetful $2$-functor $\cat_{\fc} \to \cat$ and construct a bicategorical coinserter $(P : \B \to \C,\, \alpha : P \circ F \to P \circ G)$ in $\cat$ as above. We proceed as in the proof of \cref{coprodtensor} and consider the free cocompletion $\smash{\widehat{\C}}$ with the Yoneda embedding $\smash{Y : \C \hookrightarrow \widehat{\C}}$. The problem is that the functor $\smash{Y \circ P : \B \to \widehat{\C}}$ has no reason to preserve finite colimits. To fix this, let $\Phi_\B$ denote the set of finite colimit cocones in $\B$. Then $\Phi = P(\Phi_\B)$ is a set of cocones in $\C$. Let $\smash{\Alg(\Phi) \subseteq \widehat{\C}}$ denote the full subcategory of presheaves which map all these cocones to limit cones in $\Set$. This is actually a reflective subcategory by \cite[Theorem 6.11]{Kel05} (alternatively, one may use \cite[Theorem 1.39]{AR94}), in particular cocomplete. Let $\smash{R : \widehat{\C} \to \Alg(\Phi)}$ be the reflector. Now let $\D \subseteq \Alg(\Phi)$ denote the smallest full subcategory of $\Alg(\Phi)$ which is closed under finite colimits and contains the image of $R \circ Y : \C \to \Alg(\Phi)$. Thus, $\D$ is finitely cocomplete. If $\E$ is any finitely cocomplete category, then by \cite[Theorem 6.23]{Kel05} the category $\Hom_{\fc}(\D,\E)$ is equivalent to the category of $\Phi$-comodels in $\E$, which by definition are functors $H : \C \to \E$ which map the cocones in $\Phi$ to colimit cocones in $\E$. By definition of $\C$ and $\Phi$, we arrive at the category of finitely cocontinuous functors $L : \B \to \E$ equipped with a morphism of (finitely cocontinuous) functors $L \circ F \to L \circ G$. This means that $\D$ is the desired bicategorical coinserter in $\cat_{\fc}$.
    
The construction in $\cat_{\fc\!/ \IK}$ works almost the same. We start with the bicategorical coinserter in $\cat_{\IK}$, consider its free $\IK$-linear cocompletion and proceed as above. In fact, the theory of \cite[Chapter 6]{Kel05} works for quite general enriched categories.
\end{proof}

\begin{prop} \label{coequifiers}
The $2$-categories $\cat$, $\cat_{\IK}$, $\cat_{\fc}$, $\cat_{\fc\!/\IK}$ have bicategorical coequifiers.
\end{prop}

\begin{proof}
We start with $\cat$. Let $F,G : \A \rightrightarrows \B$ be two functors between small categories and let $\alpha,\beta : F \rightrightarrows G$ be two morphisms of functors. Let $R$ be the smallest congruence relation on $\B$ which contains all pairs $(\alpha(A),\beta(A))$ of morphisms $F(A) \to G(A)$ for $A \in \Ob(\A)$. Then the quotient category $\B/R$ with the projection functor $P : \B \to \B/R$ is a $2$-categorical and hence bicategorical coequifier of $\alpha,\beta$.

For $\cat_{\IK}$ a similar construction works; here we consider of course only congruence relations which are a compatible with the $\IK$-linear structure.

For $\cat_{\fc}$ we use the same method as in the proof of \cref{coprodtensor}. With the notation above, assume that $\A,\B$ are finitely cocomplete and that $F,G$ are finitely cocontinuous (there are no additionial assumptions on $\alpha,\beta$). We construct the coequifier $P : \B \to \C$ in $\cat$ as above, define $\Phi$ as the image of all finite colimit cocones in $\B$ under $P$, define $\smash{\Alg(\Phi) \subseteq \widehat{\C}}$ as the (reflective and hence cocomplete) category of presheaves which send the cocones in $\Phi$ to limit cones, and finally define $\D$ as the closure of the image of $\smash{\C \hookrightarrow \widehat{\C} \twoheadrightarrow \Alg(\Phi)}$ under finite colimits. Then the composition $\B \to \C \to \D$ is a bicategorical coequifier of $\alpha,\beta$ in $\cat_{\fc}$.
 
For $\cat_{\fc\!/\IK}$ a similar argument works.
\end{proof}

\begin{cor} \label{catcocomplete}
The $2$-categories $\cat$, $\cat_{\IK}$, $\cat_{\fc}$, $\cat_{\fc\!/\IK}$ are bicategorically cocomplete.
\end{cor}

\begin{proof}
This follows from \cref{coprodcat}, \cref{coins} and \cref{coequifiers} using \cref{allcolim}.
\end{proof}

Next we will treat tensor categories. For now we will focus on the simplest case $\cat_{\otimes}$ and say later what needs to be changed for the other types of tensor categories (in fact, not much).
 
\begin{rem} \label{preservecolim}
The product functor $\times : \cat^2 \to \cat$ preserves bicategorical coinserters and bicategorical coequifiers in each variable, in fact all bicategorical colimits. This is because $\A \times -$ (and likewise $- \times \A$) is $2$-categorically (and hence bicategorically) left adjoint to $\HOM(\A,-)$. Applying this twice, it follows that if $\B \to \C$ is the bicategorical coinserter of $\A \rightrightarrows \B$ and $\B' \to \C'$ is the bicategorical coinserter of $\A' \rightrightarrows \B'$, then $\B \times \B' \to \C \times \C'$ is the multiple bicategorical coinserter of $\A \times \B' \rightrightarrows \B \times \B'$ and $\B \times \A' \rightrightarrows \B \times \B'$. A similar statement holds for bicategorical coequifiers.
\end{rem}

The following construction of bicategorical coequalizers (or more generally, bicategorical coinserters) of tensor categories is motivated by the construction of coequalizers of commutative monoids. If $f,g : A \rightrightarrows B$ are two homomorphisms of commutative monoids, we define auxiliary maps $\overline{f},\overline{g} : B \times A \rightrightarrows B$ by $\overline{f}(b,a) \coloneqq b \cdot f(a)$ and $\overline{g}(b,a) \coloneqq b \cdot g(a)$. If $p : B \to C$ is the coequalizer of $\overline{f},\overline{g}$ in the category of sets, then $C$ has a unique commutative monoid structure which makes $p$ a homomorphism. Moreover, $p$ is the coequalizer of $f,g$ in the category of commutative monoids. Similarly, coequalizers of commutative rings can be constructed using coequalizers of abelian groups. For tensor categories, which we can view as categorified commutative monoids, we can start with a similar construction. However, equalities have to be replaced by (iso-)morphisms, and their coherence conditions need some extra care. We will also need a second step which is not visible in the case of commutative monoids.

\begin{prop} \label{coinsish}
Let $\A$ be a small tensor category, $\I$ be a small category and $F,G : \I \rightrightarrows \A$ be two functors into (the underlying category of) $\A$. Then there is a small tensor category $\B$ with a tensor functor $P : \A \to \B$ and a morphism of functors $\delta' : P \circ F \to P \circ G$ (without tensor structure) which is universal: For every small tensor category $\C$ we get an equivalence of categories between the category of tensor functors $\B \to \C$ and the category of tensor functors $H : \A \to \C$ equipped with a morphism of functors $H \circ F \to H \circ G$.
\end{prop}

\begin{proof}
In order to reduce the formalism and keep the notation as simple as possible, we will be very sloppy and denote the components of a natural transformation just by the same symbol, even when we apply other functors to them. Also, we will usually just write $\sim$ when coherence isomorphisms in $\A$ are used.
 
We define an auxiliary functor (a ``stable'' version of $F$)
\[\overline{F} : \A \times \I \to \A,\, (A,i) \mapsto A \otimes F(i).\]
We define $\overline{G} : \A \times \I \to \A$ in the same way. Let $P : \A \to \B$ be the bicategorical coinserter of $\overline{F},\overline{G} : \A \times \I \rightrightarrows \A$ in $\cat$ (which exists by \cref{coins}); it comes equipped with a morphism of functors
\[\delta : P \circ \overline{F} \to P \circ \overline{G}.\]
Precomposing $\delta$ with the functor $\I \to \A \times \I$, $i \mapsto (1_\A,i)$ yields a morphism of functors $ \delta' : P \circ F \to P \circ G$. Recall that the definition of $P$ means (a) that for every functor $H : \A \to \C$ with a morphism $\gamma : H \circ \overline{F} \to H \circ \overline{G}$ there is a functor $K : \B \to \C$ with an isomorphism $K \circ P \myiso H$ such that $\gamma$ corresponds to $K \circ \delta$ under this isomorphism, and (b) that for two functors $K,L: \B \to \C$ any morphism $K \circ P \to L \circ P$ which is compatible with $\delta$ is induced by a unique morphism $K \to L$.
 
The next step is to define a tensor structure on $\B$. We define $1_\B \coloneqq P(1_\A)$, so that the identity is an isomorphism $\eta : 1_\B \myiso P(1_\A)$. Now consider the functor $P \circ \otimes_\A : \A^2 \to \A \to \B$, $(A,B) \mapsto P(A \otimes B)$, which we would like to extend to $\B^2$. We observe that the functor coinserts the functors $\overline{F} \times \A,\, \overline{G} \times \A : \A \times \I \times \A \rightrightarrows \A^2$ via the natural morphisms
\[\begin{tikzcd}[column sep=3ex]
P\bigl((A \otimes F(i)) \otimes B\bigr) \ar{r}{\sim} & P\bigl((B \otimes A) \otimes F(i)\bigr) \ar{r}{\delta} &   P\bigl((B \otimes A) \otimes G(i)\bigr) \ar{r}{\sim} & P\bigl((A \otimes G(i)) \otimes B\bigr).
\end{tikzcd}\]
Similarly, the functor coinserts the functors $\A \times \overline{F},\, \A \times \overline{G} : \A \times \A \times \I \rightrightarrows \A^2$ via the natural morphisms
\[\begin{tikzcd}[column sep=3ex]
P\bigl(A \otimes (B \otimes F(i))\bigr) \ar{r}{\sim} & P\bigl((A \otimes B) \otimes F(i)\bigr) \ar{r}{\delta} & P\bigl((A \otimes B) \otimes G(i)\bigr) \ar{r}{\sim} & P\bigl(A \otimes (B \otimes G(i))\bigr).
\end{tikzcd}\]
Hence, by definition of $P$ and \cref{preservecolim}, it follows that there is a functor $\otimes_\B : \B^2 \to \B$ with an isomorphism $\mu : \otimes_\B \circ P^2 \myiso P \circ \otimes_\A$, i.e.\ natural isomorphisms
\[\mu : P(A) \otimes_\B P(B) \myiso P(A \otimes_\A B)\]
for $A,B \in \A$ which together with $\delta$ induce the isomorphisms above. This means that the two following diagrams commute (we will abbreviate $\otimes_\B$ and $\otimes_\A$ by $\otimes$):

\begin{equation} \label{diag1}
\begin{tikzcd}[row sep=5ex]
P\bigl(A \otimes F(i)\bigr) \otimes P(B) \ar{r}{\mu} \ar{d}[swap]{\delta} & P\bigl((A \otimes F(i)) \otimes B\bigr) \ar{r}{\sim} & P\bigl((B \otimes A) \otimes F(i)\bigr) \ar{d}{\delta} \\
P\bigl(A \otimes G(i)\bigr) \otimes P(B) \ar{r}[swap]{\mu} & P\bigl((A \otimes G(i)) \otimes B\bigr) \ar{r}[swap]{\sim} & P\bigl((B \otimes A) \otimes G(i)\bigr)
\end{tikzcd}
\end{equation}
\vphantom{I am some space}
\begin{equation} \label{diag2}
\begin{tikzcd}[row sep=5ex]
P(A) \otimes P\bigl(B \otimes F(i)\bigr) \ar{r}{\mu} \ar{d}[swap]{\delta} & P\bigl(A \otimes (B \otimes F(i))\bigr) \ar{r}{\sim} & P\bigl((A \otimes B) \otimes F(i)\bigr) \ar{d}{\delta} \\
P(A) \otimes P\bigl(B \otimes G(i)\bigr) \ar{r}[swap]{\mu} & P\bigl(A \otimes (B \otimes G(i))\bigr) \ar{r}[swap]{\sim} & P\bigl((A \otimes B) \otimes G(i)\bigr) 
\end{tikzcd}
\end{equation}
In order to define the natural coherence isomorphisms $\rho : U \otimes P(1_\A) \myiso U$, $\lambda : P(1_\A) \otimes U \myiso U$ for $ U \in \B$, it suffices (by part (b) in the definition of $P$) to construct natural isomorphisms $\rho' : P(A) \otimes P(1_\A) \myiso P(A)$, $\lambda' : P(1_\A) \otimes P(A) \myiso P(A)$ for $A \in \A$ and check their compatibility with $\delta$. We define them to be (using the coherence isomorphisms in $\A$)
\[\begin{tikzcd}
P(A) \otimes P(1_\A) \ar{r}{\mu} & P(A \otimes 1_\A) \ar{r}{\rho} &  P(A), \\[-2ex]
P(1_\A) \otimes P(A)  \ar{r}{\mu} & P(1_\A \otimes A) \ar{r}{\lambda} & P(A).
\end{tikzcd}\]
This is actually the only choice we have since we want $P$ to become a tensor functor later. The compatibility of $\rho'$ with $\delta$ follows from the following commutative diagram:
\[\begin{tikzcd}[row sep = 6ex, column sep = 7ex]
P\bigl(A \otimes F(i)\bigr) \otimes P(1_\A) \ar{r}{\mu} \ar{ddd}[swap]{\delta} & P\bigl((A \otimes F(i)) \otimes 1_\A\bigr) \ar{r}{\rho} \ar{d}[swap]{\sim} &  P\bigl(A \otimes F(i)\bigr) \ar{ddd}{\delta} \\
& P\bigl((1_\A \otimes A) \otimes F(i)\bigr) \ar{d}[swap]{\delta} \ar{ur}[swap]{\lambda} & \\
& P\bigl((1_\A \otimes A) \otimes G(i)\bigr) \ar{dr}{\lambda} &  \\
P\bigl(A \otimes G(i)\bigr) \otimes P(1_\A) \ar{r}[swap]{\mu} & P\bigl((A \otimes G(i)) \otimes 1_\A\bigr) \ar{r}[swap]{\rho} \ar{u}{\sim} &  P\bigl(A \otimes G(i)\bigr)
\end{tikzcd}\]
Here, the left rectangle is diagram (\ref{diag1}) in a special case, the two triangles commute because of coherence in $\A$, and the trapezoid commutes because of naturality of $\delta$. The compatibility of $\lambda'$ with $\delta$ follows from a similar commutative diagram:
\[\begin{tikzcd}[row sep = 6ex, column sep = 7ex]
P(1_\A) \otimes P\bigl(A \otimes F(i)\bigr) \ar{r}{\mu} \ar{ddd}[swap]{\delta} & P\bigl(1_\A \otimes (A \otimes F(i))\bigr) \ar{r}{\lambda} \ar{d}[swap]{\sim} & P\bigl(A \otimes F(i)\bigr) \ar{ddd}{\delta} \\
& P\bigl((1_\A \otimes A) \otimes F(i)\bigr) \ar{d}[swap]{\delta} \ar{ur}[swap]{\lambda} \\
& P\bigl((1_\A \otimes A) \otimes G(i)\bigr)  \ar{dr}{\lambda} \\
P(1_\A) \otimes P\bigl(A \otimes G(i)\bigr) \ar{r}[swap]{\mu}  & P\bigl(1_\A \otimes (A \otimes G(i))\bigr) \ar{r}[swap]{\lambda} \ar{u}{\sim} & P\bigl(A \otimes G(i)\bigr) 
\end{tikzcd}\]
Here, the left rectangle is diagram (\ref{diag2}) in a special case, the two triangles commute because of coherence in $\A$, and the trapezoid commutes because of naturality of $\delta$.
 
In order to define the associator $\alpha$ on $\B$, it suffices to construct natural isomorphisms
\[\alpha' : \bigl(P(A) \otimes P(B)\bigr) \otimes P(C) \myiso P(A) \otimes \bigl(P(B) \otimes P(C)\bigr)\]
for $A,B,C \in \A$ and check their compatibility with $\delta$ in all three variables. Since we want $P$ to become a tensor functor, we have no choice but to define $\alpha'$ by the commutativity of
\[\begin{tikzcd}[row sep = 5ex]
\bigl(P(A) \otimes P(B)\bigr) \otimes P(C) \ar{d}[swap]{\mu} \ar{r}{\alpha'} & P(A) \otimes \bigl(P(B) \otimes P(C)\bigr)  \ar{d}{\mu} \\ P(A \otimes B) \otimes P(C) \ar{d}[swap]{\mu} & P(A) \otimes P(B \otimes C) \ar{d}{\mu} \\ P\bigl((A \otimes B) \otimes C\bigr) \ar{r}[swap]{\alpha}   & P\bigl(A \otimes (B \otimes C)\bigr).
\end{tikzcd}\]
The compatibility of $\alpha'$ in the first variable follows from the following commutative diagram; for simplicity we have dropped all $\otimes$-symbols and have replaced $F(i)$ by $F$.
\[\begin{tikzcd}[row sep = 5ex, column sep = 4ex]
\bigl(P(A F) P(B)\bigr) P(C) \ar{rrr}{\delta} \ar{d}[swap]{\mu} & & & \bigl(P(A G) P(B)\bigr) P(C) \ar{d}{\mu} \\
P\bigl((A F) B\bigr) P(C) \ar{r}{\sim} \ar{d}[swap]{\mu} & P\bigl((BA) F\bigr) P(C) \ar{r}{\delta} \ar{d}[swap]{\mu} & P\bigl((BA) G\bigr) P(C) \ar{d}{\mu} &  P\bigl((A G) B\bigr) P(C) \ar{l}[swap]{\sim} \ar{d}{\mu} \\
P\bigl(((A F) B) C\bigr) \ar{dd}[swap]{\sim} \ar{r}{\sim} & P\bigl(((BA) F) C\bigr) \ar{d}[swap]{\sim} & P\bigl(((BA) G) C\bigr) \ar{d}{\sim}  & P\bigl(((A G) B) C\bigr) \ar{dd}{\sim} \ar{l}[swap]{\sim} \\
 & P\bigl((C(BA)) F\bigr) \ar{r}[swap]{\delta} \ar{d}[swap]{\sim} & P\bigl((C(BA)) G\bigr) \ar{d}{\sim} & \\
P\bigl((A F) (B C)\bigr) \ar{r}{\sim} & P\bigl(((BC)A) F\bigr) \ar{r}[swap]{\delta} & P\bigl(((BC)A) G\bigr)  & P\bigl((A G) (B C)\bigr) \ar{l}[swap]{\sim}  \\
P(A F) P(B C) \ar{u}{\mu} \ar{rrr}[swap]{\delta} &&& P(A G) P(B C) \ar{u}[swap]{\mu} \\
P(A F) \bigl(P(B) P(C)\bigr) \ar{u}{\mu} \ar{rrr}[swap]{\delta} &&& P(A G) \bigl(P(B) P(C)\bigr) \ar{u}[swap]{\mu}
\end{tikzcd}\]
Here, we have used diagram (\ref{diag1}) three times, two times coherence in $\A$, and the rest commutes because of naturality. The compatibility of $\alpha'$ in the third variable follows from a very similar diagram (which is just upside down) which uses diagram (\ref{diag2}) three times. The compatibility of $\alpha'$ in the second variable follows from a different diagram:
\[\begin{tikzcd}[row sep = 5ex, column sep = 4ex]
\bigl(P(A) P(BF)\bigr) P(C) \ar{d}[swap]{\mu}\ar{rrr}{\delta} &&& P(A) \bigl( P(BG) P(C) \bigr) \ar{d}{\mu} \\
P\bigl(A(BF)\bigr) P(C) \ar{d}[swap]{\mu} \ar{r}{\sim} & P\bigl((AB)F\bigr) P(C) \ar{d}[swap]{\mu} \ar{r}{\delta} & P\bigl((AB)G\bigr) P(C) \ar{d}{\mu} &  P\bigl(A(BG)\bigr) P(C) \ar{l}[swap]{\sim} \ar{d}{\mu} \\
P\bigl((A(BF))C\bigr) \ar{ddd}[swap]{\sim} \ar{r}{\sim} & P\bigl(((AB)F)C\bigr) \ar{d}[swap]{\sim} & P\bigl(((AB)G)C\bigr) \ar{d}{\sim} & P\bigl((A(BG))C\bigr) \ar{l}[swap]{\sim} \ar{ddd}{\sim} \\
& P\bigl((C(AB))F\bigr)  \ar{d}[swap]{\sim} \ar{r}{\delta} & P\bigl((C(AB))G\bigr) \ar{d}{\sim}& \\
& P\bigl((A(CB))F\bigr) \ar{r}[swap]{\delta} & P\bigl((A(CB))G\bigr) & \\
P\bigl(A((BF)C)\bigr) \ar{r}[swap]{\sim} & P\bigl(A((CB)F)\bigr) \ar{u}{\sim} & P\bigl(A((CB)G)\bigr) \ar{u}[swap]{\sim} & \ar{l}{\sim} P\bigl(A((BG)C)\bigr)\\
P(A) P\bigl((BF)C\bigr) \ar{u}{\mu} \ar{r}[swap]{\sim} & P(A) P\bigl((CB)F\bigr) \ar{r}[swap]{\delta} \ar{u}{\mu} & P(A) P\bigl((CB)G\bigr) \ar{u}[swap]{\mu} & P(A) P\bigl((BG)C\bigr) \ar{u}[swap]{\mu} \ar{l}{\sim} \\
P(A) \bigl(P(BF) P(C)\bigr) \ar{u}{\mu} \ar{rrr}[swap]{\delta} &&& P(A) \bigl(P(BG) P(C)\bigr) \ar{u}[swap]{\mu}
\end{tikzcd}\]
Here, we have used each of the diagrams (\ref{diag1}), (\ref{diag2}) twice, two times coherence in $\A$, and the rest commutes because of naturality. This finishes the construction of the associator $\alpha$ on $\B$.

In order to define the symmetry $\sigma$ on $\B$, it is sufficient to define natural isomorphisms $\sigma' : P(A) \otimes P(B) \to P(B) \otimes P(A)$ for $A,B \in \A$ and check their compatibility with $\delta$ in both variables. The only reasonable choice is (using the symmetry $\sigma$ on $\A$)
\[\begin{tikzcd}
P(A) \otimes P(B) \ar{r}{\mu} & P(A \otimes B) \ar{r}{\sigma} & P(B \otimes A) \ar{r}{\mu^{-1}} & P(B) \otimes P(A).
\end{tikzcd}\]
The compatibility of $\sigma'$ in the first variable follows from the following commutative diagram:
\[\begin{tikzcd}[row sep=6ex, column sep=4ex]
   & P((AF)B) \ar{dr}[swap]{\sim} \ar{rr}{\sigma} && P(B(AF))  \ar{dl}{\sim} & \\
P(A F) P(B) \ar{ur}{\mu} \ar{d}[swap]{\delta} && P((BA)F) \ar{d}{\delta} && P(B) P(AF) \ar{ul}[swap]{\mu} \ar{d}{\delta}  \\
P(A G) P(B) \ar{dr}[swap]{\mu} && P((BA)G)  && P(B) P(AG) \ar{dl}{\mu}  \\
  & P((AG)B) \ar{rr}[swap]{\sigma} \ar{ur}{\sim} && P(B(AG)) \ar{ul}[swap]{\sim} & 
\end{tikzcd}\]
Here, the two triangles commute because of coherence in $\A$, and the two hexagons commute because of diagrams (\ref{diag1}) and (\ref{diag2}). The compatibility of $\sigma'$ in the second variable follows from a similiar commutative diagram, which we omit here.
 
In order to check the coherence diagrams in the definition of a tensor category, it suffices to check them when composed with $P$. But then, since by construction $P$ is compatible with the coherence isomorphisms $\lambda,\rho,\alpha,\sigma$, we can just use the coherence diagrams in $\A$ and are done. Thus, $\B$ with the given data becomes a tensor category, and by construction $P : \A \to \B$ becomes a tensor functor. (To be precise, the tuple $(\B,\otimes,\lambda,\rho,\alpha,\sigma)$ is a tensor category and the tuple $(P,\eta,\mu)$ is a tensor functor, but we will use this common abuse of notation.)
 
As already mentioned in the beginning, there is a morphism of functors $\delta' : P \circ F \to P \circ G$, and what is left is to prove the universal property of $P$ as a tensor functor, i.e.\ that for every small tensor category $\C$ the functor
\[Q : \Hom_{\otimes}(\B,\C) \to \{(H,\gamma) : H \in \Hom_{\otimes}(\A,\C),\, \gamma : H \circ F \to H \circ G\},\, K \mapsto (K \circ P, K \circ \delta')\]
is an equivalence of categories.

The functor $Q$ is fully faithful: Let $K,L : \B \to \C$ be two tensor functors with a morphism $\vartheta : (K \circ P,K \circ \delta') \to (L \circ P,L \circ \delta')$ in the target category, i.e.\ $\vartheta : K \circ P \to L \circ P$ is a morphism of tensor functors which is compatible with $\delta'$, which means that
\begin{equation} \label{diag3}
\begin{tikzcd}[row sep = 5ex]
K \circ P \circ F \ar{r}{\vartheta} \ar{d}[swap]{\delta'} & L \circ P \circ F \ar{d}{\delta'} \\
K \circ P \circ G \ar{r}[swap]{\vartheta} & L \circ P \circ G
\end{tikzcd}
\end{equation}
commutes. Then $\vartheta$ is even compatible with $\delta$, i.e.\ also the diagram
\begin{equation} \label{diag4}
\begin{tikzcd}[row sep = 5ex]
K \circ P \circ \overline{F} \ar{r}{\vartheta} \ar{d}[swap]{\delta} & L \circ P \circ \overline{F} \ar{d}{\delta} \\
K \circ P \circ \overline{G} \ar{r}[swap]{\vartheta} & L \circ P \circ \overline{G}
\end{tikzcd}
\end{equation}
commutes: To see this, first notice that diagram (\ref{diag2}) implies that
\begin{equation} \label{diag5}
\begin{tikzcd}[row sep = 5ex]
P(A) \otimes P(F(i)) \ar{d}[swap]{\delta'} \ar{r}{\mu} & P(A \otimes F(i)) \ar{d}{\delta} \\
P(A) \otimes P(G(i)) \ar{r}[swap]{\mu} & P(A \otimes G(i))
\end{tikzcd}
\end{equation}
commutes. Now consider the following diagram (for better readability, we have omitted some brackets and the index $i$):
\[\begin{tikzcd}[column sep = 1ex, row sep = 3ex]
KPA \otimes KPF \ar{ddd}[swap]{\delta'} \ar{rrr}{\vartheta \otimes \vartheta}\ar{dr}{\sim}  &&& LPA \otimes LPF \ar{dr}{\sim}\ar{ddd}{\delta'} && \\
& K(PA \otimes PF) \ar{dr}{\mu} &&& L(PA \otimes PF) \ar{dr}{\mu} \ar{ddd}{\delta'}& \\
&& KP(A \otimes F)  &&& LP(A \otimes F) \ar{ddd}{\delta} \ar[from=lll,crossing over,"\vartheta"]\\
KPA \otimes KPG \ar{rrr}[swap]{\vartheta \otimes \vartheta} \ar{dr}[swap]{\sim} &&& LPA \otimes LPG \ar{dr}[swap]{\sim} && \\
& K(PA \otimes PG) \ar{dr}[swap]{\mu} \ar[from=uuu,crossing over,swap,"\delta'"] &&& L(PA \otimes PG) \ar{dr}[swap]{\mu} & \\
&& KP(A \otimes G) \ar{rrr}[swap]{\vartheta} \ar[from=uuu,crossing over,swap,"\delta"] &&& LP(A \otimes G)
\end{tikzcd}
\]
The bottom and top faces commute because $\vartheta$ is a morphism of tensor functors, the two squares on the left face commute because of naturality and of diagram (\ref{diag5}), the same holds for the right face, and the back face commutes because of diagram (\ref{diag3}). Hence, the front face commutes as well, which is exactly diagram (\ref{diag4}).

Since $\vartheta$ is compatible with $\delta$, by part (b) in the definition of $P$ there is a unique morphism of functors $\pi : K \to L$ with $\vartheta = \pi \circ P$. Since $\vartheta$ is a morphism of tensor functors, it follows easily from part (b) in the definition of $P$ that $\pi$ is a morphism of tensor functors as well. Thus, there is a unique morphism of tensor functors $\pi : K \to L$ with $\vartheta = \pi \circ P$, which finishes the proof of fully faithfulness.
 
The functor $Q$ is essentially surjective: Let $H : \A \to \C$ be a tensor functor with a morphism of functors $\gamma : H \circ F \to H \circ G$. We can extend it to a morphism of functors $\tilde{\gamma} : H \circ \overline{F} \to H \circ \overline{G}$ by
\[\begin{tikzcd}
H(A \otimes F(i)) \ar{r}{\sim} & H(A) \otimes H(F(i)) \ar{r}{\gamma} & H(A) \otimes H(G(i)) \ar{r}{\sim} & H(A \otimes G(i)).
\end{tikzcd}\]
By part (a) in the definition of $P$ there is a functor $K : \B \to \C$ with an isomorphism of functors $\varepsilon : K \circ P \myiso H$ such that
\begin{equation} \label{diag6}
\begin{tikzcd}[row sep = 5ex]
K \circ P \circ \overline{F} \ar{d}[swap]{\delta} \ar{r}{\varepsilon} & H \circ \overline{F} \ar{d}{\tilde{\gamma}} \\
K \circ P \circ \overline{G} \ar{r}[swap]{\varepsilon} & H \circ \overline{G}
\end{tikzcd}
\end{equation}
commutes. Then clearly also
\[\begin{tikzcd}[row sep = 5ex]
K \circ P \circ F \ar{d}[swap]{\delta'} \ar{r}{\varepsilon} & H \circ F \ar{d}{\gamma} \\
K \circ P \circ G \ar{r}[swap]{\varepsilon} & H \circ G
\end{tikzcd}\]
commutes.

We now have to make $K$ into a tensor functor in such a way that $\varepsilon$ becomes an isomorphism of tensor functors. Again, we have only one choice. We have the isomorphism
\[1_\C \myiso H(1_\A) \xrightarrow{\varepsilon^{-1}} K(P(1_\A)) = K(1_\B).\]
In order to construct natural isomorphisms $K(U) \otimes K(V) \myiso K(U \otimes V)$ for $U,V \in \B$, it suffices to construct natural isomorphisms $K(P(A)) \otimes K(P(B)) \myiso K(P(A) \otimes P(B))$ for $A,B \in \A$ which are compatible with $\delta$ in both variables. We define them by
\[K(P(A)) \otimes K(P(B)) \xrightarrow{\varepsilon \otimes \varepsilon} H(A) \otimes H(B) \myiso H(A \otimes B) \xrightarrow{\varepsilon^{-1}\!} K(P(A \otimes B)) \xrightarrow{\mu^{-1}\!} K\bigl(P(A) \otimes P(B)\bigr).\]
The following diagram shows the compatibility of these isomorphisms with $\delta$ in the first variable; the diagram for the second variable is similar. Again, we simplify the notation.
\[\begin{tikzcd}[column sep = 3ex, row sep = 5ex]
KP(A \otimes F) \otimes KPB \ar{rrr}{\delta} \ar{d}[swap]{\varepsilon \otimes \varepsilon} &&& KP(A \otimes G) \otimes KPB \ar{d}{\varepsilon \otimes \varepsilon} \\
H(A \otimes F) \otimes HB \ar{rrr}{\tilde{\gamma}} \ar{ddd}[swap]{\sim} &&& H(A \otimes G) \otimes HB \ar{ddd}{\sim} \\
& (HA \otimes HF) \otimes HB \ar{r}{\gamma} \ar{ul}[swap]{\sim} \ar{d}[swap]{\sim} & (HA \otimes HG) \otimes HB \ar{ur}{\sim} \ar{d}{\sim}  & \\
& H(B \otimes A) \otimes HF \ar{r}{\gamma} \ar{d}[swap]{\sim} & H(B \otimes A) \otimes HG \ar{d}{\sim} & \\
H\bigl((A \otimes F) \otimes B\bigr) \ar{r}{\sim} \ar{d}[swap]{\varepsilon^{-1}} & H\bigl((B \otimes A) \otimes F\bigr) \ar{r}{\tilde{\gamma}} \ar{d}[swap]{\varepsilon^{-1}} & H\bigl((B \otimes A) \otimes G\bigr) \ar{d}{\varepsilon^{-1}} & H\bigl((A \otimes G) \otimes B\bigr) \ar{l}[swap]{\sim} \ar{d}{\varepsilon^{-1}} \\
KP\bigl((A \otimes F) \otimes B\bigr) \ar{r}{\sim} \ar{d}[swap]{\mu^{-1}} & KP\bigl((B \otimes A) \otimes F\bigr) \ar{r}{\delta} & KP\bigl((B \otimes A) \otimes G\bigr) & KP\bigl((A \otimes G) \otimes B\bigr) \ar{l}[swap]{\sim} \ar{d}{\mu^{-1}} \\
K\bigl(P(A \otimes F) \otimes PB\bigr) \ar{rrr}[swap]{\delta} &&& K\bigl(P(A \otimes G) \otimes PB\bigr)
\end{tikzcd}\]
The rectangle on the top commutes because of diagram (\ref{diag6}), the trapezoid underneath it commutes by the definition of $\tilde{\gamma}$, the rectangle underneath it for trivial reasons, the rectangle underneath it by the definition of $\tilde{\gamma}$, the rectangle underneath it by diagram (\ref{diag6}). The two other trapezoids commute because of coherence of the tensor structure on $H$, the rectangles underneath them commute because of naturality. The rectangle on the bottom commutes because of diagram (\ref{diag1}).

This finishes the construction of the natural isomorphisms $K(U) \otimes K(V) \myiso K(U \otimes V)$ for $U,V \in \B$. The coherence diagrams for the tensor structure on $K$ follow from the corresponding diagrams for $H$. By construction of the tensor structure on $K$ the isomorphism $\varepsilon : K \circ P \myiso H$ is compatible with the tensor structures. This finishes the proof.
\end{proof}

\begin{rem} \label{coisoinsish}
With the notation of \cref{coinsish}, a similar construction gives a universal tensor functor $P : \C \to \D$ with an \emph{isomorphism} $P \circ F \myiso P \circ G$ of functors, using the bicategorical coequalizer (i.e.\ coisoinserter) instead of the bicategorical coinserter of $\overline{F},\overline{G}$ in $\cat$.
\end{rem}

\begin{rem} \label{monoidalcoinsish}
There is a similar construction in the category of small monoidal categories. Namely, if $\I$ is a small category and $F,G : \I \rightrightarrows \A$ are two functors into a small monoidal category $\A$, we define the auxiliary functor $\overline{F} : \A \times \I \times \A \to \A$ by $\overline{F}(A,i,B) \coloneqq A \otimes F(i) \otimes B$, similarly $\overline{G}$. Then one can endow the bicategorical coinserter $P : \A \to \B$ of $\overline{F},\overline{G}$ in $\cat$ with a monoidal structure etc.\ The proof is very similar to the one of \cref{coinsish}. It is remarkable that even when $\A$ is a symmetric monoidal category (i.e.\ a tensor category) in this setting, it is not possible to define a symmetry $\sigma$ on $\B$, since it is not clear why the auxiliary maps $\sigma' : P(A) \otimes P(B) \myiso P(B) \otimes P(A)$ are compatible with $\delta : P \circ \overline{F} \to P \circ \overline{G}$ in this case. Notice that this problem is not visible in the decategorified setting of commutative monoids, where the forgetful functor from commutative monoids to monoids creates coequalizers.
\end{rem}

\begin{rem} \label{twosteps}
Notice that \cref{coinsish} is not sufficient for the construction of bicategorical coinserters (or coequalizers, cf.\ \cref{coisoinsish}) in $\cat_{\otimes}$, since even when $\I$ is a tensor category and $F,G : \I \rightrightarrows \A$ are tensor functors, there is no reason why $\delta' : P \circ F \to P \circ G$ should be a morphism of \emph{tensor functors}. In fact, we will need coequifiers to fix this. Notice that this second step is not necessary in the decategorified setting of commutative monoids.
\end{rem}
  
\begin{prop} \label{coequifish}
Let $\A$ be a small tensor category, $\I$ be a small category and $F,G : \I \rightrightarrows \A$ be two functors into (the underlying category of) $\A$. Let $\alpha,\beta : F \rightrightarrows G$ be two morphisms of functors. Then there is a small tensor category $\B$ with a tensor functor $P : \A \to \B$ satisfying $P \circ \alpha = P \circ \beta$ which is universal: For every small tensor category $\C$ we get an equivalence of categories between the category of tensor functors $\B \to \C$ and the category of tensor functors $H : \A \to \C$ satisfying $H \circ \alpha = H \circ \beta$.
\end{prop}

\begin{proof}
As in the proof of \cref{coinsish}, we define an auxiliary functor $\overline{F} : \A \times \I \to \A$ by $\overline{F}(A,i) \coloneqq A \otimes F(i)$, similarly $\overline{G} : \A \times \I \to \A$. We also define a morphism of functors
\[\overline{\alpha} : \overline{F} \to \overline{G},\, \overline{\alpha}(A,i) \coloneqq A \otimes \alpha(i),\]
similarly $\overline{\beta}$. By \cref{coequifiers} these morphisms $\overline{\alpha},\overline{\beta}$ have a coequifier $P : \A \to \B$ in $\cat$. This means that for every small category $\C$ the functor $P$ induces an equivalence of categories between functors $\B \to \C$ and those functors $\A \to \C$ which coequify $\overline{\alpha},\overline{\beta}$. By \cref{preservecolim} it follows more generally that functors $\B^n \to \C$ correspond to functors $\A^n \to \C$ which coequify $\overline{\alpha},\overline{\beta}$ in each variable.

Our next task is to define a tensor structure on $\B$ and $P$. We define $1_\B \coloneqq P(1_\A)$ and let $\eta : 1_\B \myiso P(1_\A)$ be the identity. The functor $P \circ \otimes_\A : \A^2 \to \B$, $(A,B) \mapsto P(A \otimes B)$ coequifies $\overline{\alpha},\overline{\beta}$ in each variable, i.e.\ for all $A,B \in \A$ and $i \in \I$ we have
\begin{align*}
P(\overline{\alpha}(A,i) \otimes B) &= P(\overline{\beta}(A,i) \otimes B),\\
P(A \otimes \overline{\alpha}(B,i)) &= P(A \otimes \overline{\beta}(B,i)).
\end{align*}
This follows easily from the following commutative diagrams for $\overline{\alpha}$ and the corresponding ones for $\overline{\beta}$.
\[\begin{tikzcd}[row sep = 7ex]
(A \otimes F(i)) \otimes B \ar{r}{\sim} \ar{d}[swap]{\overline{\alpha}(A,i) \otimes B} & (A \otimes B) \otimes F(i) \ar{d}[description]{\overline{\alpha}(A \otimes B,i)} & A \otimes (B \otimes F(i)) \ar{d}{A \otimes \overline{\alpha}(B,i)} \ar{l}[swap]{\sim} \\
(A \otimes G(i)) \otimes B \ar{r}[swap]{\sim}  & (A \otimes B) \otimes G(i) &  A \otimes (B \otimes G(i)) \ar{l}{\sim}
\end{tikzcd}
\]
Hence, there is a functor $\otimes_\B : \B^2 \to \B$ with an isomorphism $\mu : \otimes_\B \circ P^2 \myiso P \circ \otimes_\A$, i.e.\ natural isomorphisms $\mu : P(A) \otimes P(B) \myiso P(A \otimes B)$.

In order to construct natural isomorphisms $\rho : U \otimes P(1_\A) \myiso U$ for $U \in \B$, it suffices to construct natural isomorphisms $\rho' : P(A) \otimes P(1_\A) \myiso P(A)$ for $A \in \A$ (here, in contrast to the proof of \cref{coinsish}, we do not have to verify any further condition), which we choose of course to be the composition of $\mu : P(A) \otimes P(1_\A) \myiso P(A \otimes 1_\A)$ and $P(\rho) : P(A \otimes 1_\A) \myiso P(A)$. In a similar way we can construct the other coherence isomorphisms $\lambda : P(1_\A) \otimes U \myiso U$, $\alpha : (U \otimes V) \otimes W \myiso U \otimes (V \otimes W)$ and $\sigma : U \otimes V \myiso V \otimes U$ for $U,V,W \in \B$.

The coherence diagrams in $\B$ follow immediately from the ones in $\A$, since in general two morphisms of functors on $\B^n$ are equal if they are equal on $\A^n$ after precomposing with $P^n$. This way $\B$ becomes a tensor category, and by construction $P : \A \to \B$ becomes a tensor functor with the isomorphisms $\eta,\mu$.

We have $P \circ \alpha = P  \circ \beta$ because of $P \circ  \overline{\alpha} = P  \circ \overline{\beta}$ and the following commutative diagram for $\overline{\alpha}$ which likewise holds for $\overline{\beta}$.
\[\begin{tikzcd}[row sep = 6ex]
1_\A \otimes F(i) \ar{r}{\sim} \ar{d}[swap]{\overline{\alpha}(1_\A,i)} & F(i) \ar{d}{\alpha(i)} \\
1_\A \otimes G(i) \ar{r}[swap]{\sim} & G(i)
\end{tikzcd}\]

Let us show the universal property of $P$ in $\cat_{\otimes}$. We start with fully faithfulness. Let $ K,L : \B \to \C$ be two tensor functors and consider a morphism $K \circ P \to L \circ P$ of tensor functors. We have to show that it is induced by a unique morphism of tensor functors $K \to L$. By the universal property of $P$ in $\cat$, it is induced by a unique morphism of functors $K \to L$, so we just have to show its compatibility with the tensor structure. We only show compatibility with binary tensor products, the compatibility with the unit objects is also easy. By the universal property of $P$, it suffices to show that the left square in the diagram
\[\begin{tikzcd}[row sep = 5ex]
K(P(A)) \otimes K(P(B)) \ar{r}{\sim} \ar{d} & K\bigl(P(A) \otimes P(B)\bigr) \ar{d} \ar{r}{\sim} & K(P(A \otimes B)) \ar{d} \\
L(P(A)) \otimes L(P(B)) \ar{r}[swap]{\sim} & L\bigl(P(A) \otimes P(B)\bigr) \ar{r}[swap]{\sim} & L(P(A \otimes B))
\end{tikzcd}\]
commutes for all $A,B \in \A$. The outer rectangle commutes because $K \circ P \to L \circ P$ is a morphism of tensor functors by assumption. The right triangle commutes since $K \to L$ is natural, so we are done.

Finally, let $H : \A \to \C$ be a tensor functor with $H  \circ \alpha = H  \circ \beta$. Then we actually have $H  \circ \overline{\alpha} = H \circ \overline{\beta}$, which follows from the following commutative diagram for $\overline{\alpha}$, likewise for $\overline{\beta}$.
\[\begin{tikzcd}[row sep = 6ex]
H(A) \otimes H(F(i)) \ar{d}[swap]{H(A) \otimes H(\alpha(i))} \ar{r}{\sim} & H(A \otimes F(i)) \ar{d}{H(\overline{\alpha}(A,i))} \\
H(A) \otimes H(G(i)) \ar{r}[swap]{\sim} & H(B \otimes G(i))
\end{tikzcd}\]
Hence, there is a functor $L : \B \to \C$ with an isomorphism of functors $L \circ P \myiso H$. We have to endow $L$ with the structure of a tensor functor. In fact, we have an isomorphism $1_\C \myiso H(1_\A) \myiso L(P(1_\A)) \myiso L(1_\B)$, and for $A,B \in \A$ we have natural isomorphisms
\[
L(P(A)) \otimes L(P(B)) \myiso  H(A) \otimes H(B) \myiso  H(A \otimes B) \myiso  L(P(A \otimes B)) \myiso  L(P(A) \otimes P(B)),
\]
which thus induce natural isomorphisms $L(U) \otimes L(V) \myiso L(U \otimes V)$ for $U,V \in \B$. The coherence diagrams in the definition of a tensor functor follow for $L$ immediately from those for $H$. By construction $L \circ P \myiso H$ is actually an isomorphism of tensor functors.
\end{proof}

\begin{cor} \label{tensorcoequifiers}
The $2$-category $\cat_{\otimes}$ has bicategorical coequifiers.
\end{cor}

\begin{proof}
This follows immediately from \cref{coequifish}, since two morphisms of tensor functors are equal if and only if they are equal as morphisms of functors.
\end{proof}

\begin{prop} \label{tensorcoins}
The $2$-category $\cat_{\otimes}$ has bicategorical coinserters.
\end{prop}

\begin{proof}
Let $F,G : \A \rightrightarrows \B$ be two tensor functors between small tensor categories. By \cref{coinsish} there is a bicategorically universal tensor functor $P : \B \to \C$ with a morphism of functors $\delta' : P \circ F \to P \circ G$. It is not necessarily a morphism of tensor functors. In order to fix this, consider the parallel pair of morphisms of functors
\[\begin{tikzcd}
\alpha,\beta : P\bigl(F(-)\bigr) \otimes P\bigl(F(-)\bigr) \ar[shift left=0.6ex]{r} \ar[shift right=0.6ex]{r} &  P\bigl(G(- \otimes -)\bigr) : \A^2 \to \C
\end{tikzcd}\]
given by
\[\begin{tikzcd}
\alpha : P\bigl(F(-)\bigr) \otimes P\bigl(F(-)\bigr) \ar{r}{\sim} & P\bigl(F(- \otimes -)\bigr) \ar{r}{\delta'} & P\bigl(G(- \otimes -)\bigr)
\end{tikzcd}\]
and
\[\begin{tikzcd}
\beta : P\bigl(F(-)\bigr) \otimes P\bigl(F(-)\bigr) \ar{r}{\delta' \otimes \delta'} &  P\bigl(G(-)\bigr) \otimes P\bigl(G(-)\bigr) \ar{r}{\sim} &  P\bigl(G(- \otimes -)\bigr).
\end{tikzcd}\]
Similarly, consider the parallel pair of morphisms of functors
\[\begin{tikzcd}
\alpha',\beta' : 1_\C \ar[shift left=0.6ex]{r} \ar[shift right=0.6ex]{r} & P(G(1_\A)) : \A^0 \to \C
\end{tikzcd}\]
given by $\alpha' : 1_\C \myiso P(F(1_\A)) \xrightarrow{\delta'} P(G(1_\A))$ and $\beta' : 1_\C \myiso P(G(1_\A))$. We apply \cref{coequifish} twice to obtain a bicategorically universal tensor functor $Q : \C \to \D$ with $Q \circ \alpha = Q \circ \beta$ and $Q \circ \alpha' = Q \circ \beta'$. Then $R \coloneqq Q \circ P : \B \to \D$ is a bicategorically universal tensor functor with a morphism of \emph{tensor functors} $R \circ F \to R \circ G$, namely $Q \circ \delta'$.
\end{proof}

\begin{cor} \label{normaltensorpush}
The $2$-category $\cat_{\otimes}$ has bicategorical pushouts.
\end{cor}

\begin{proof}
This follows from \cref{tensorcoequifiers}, \cref{tensorcoins}, \cref{coprodtensor} and \cref{pushviacoprodcoeqcoins}. Let us briefly spell out how bicategorical pushouts actually look like according to our proofs. Given tensor functors $F : \C \to \A$, $G : \C \to \B$, we form the bicategorical coproduct of $A,B$ in $\cat_{\otimes}$, which is just the product $\A \times \B$. Then we form the bicategorical coequalizer $ P : \A \times \B \to \D$ in $\cat$ of the parallel pair of functors $\A \times \B \times \C \rightrightarrows \A \times \B$ defined by $(A,B,C) \mapsto (A \otimes F(C),B)$ and $(A,B,C) \mapsto (A,B \otimes G(C))$. Actually $P$ carries a tensor structure, but the natural isomorphisms $P(A \otimes F(C),B) \myiso P(A , B \otimes G(C))$ need not be compatible with the tensor structure. A suitable multiple bicategorical coequifier $Q : \D \to \E$ in $\cat$ rectifies this, which actually carries a tensor structure, and $\E$ is the desired bicategorical pushout in $\cat_{\otimes}$.
\end{proof}

\begin{rem} \label{codescent}
Let us say a few words on a different construction of bicategorical pushouts in $\cat_{\otimes}$ (which works similarly for the other variants), which was indicated by Schäppi in \cite[Section 4]{Sch18} and involves more category theoretic machinery, even though there are obvious similarities to our construction above. Given tensor functors $\A \leftarrow \C \rightarrow \B$, one constructs \emph{reflexive coherence data} \cite{Lac02} (a truncated simplicial object in which the simplicial identities only hold up to specified isomorphisms which are coherent with each other)
\[\begin{tikzcd}[row sep=6ex, column sep = 8ex]
\A \times \C \times \C \times \B \ar[shift left=2ex]{r} \ar{r} \ar[shift right=2ex]{r} &
\A \times \C \times \B \ar[shift left=1ex]{r} \ar[shift right=1ex]{r} \ar[shift left=1ex]{l} \ar[shift right=1ex]{l} & 
\A \times \B. \ar{l}
\end{tikzcd}\]
They have a \emph{bicategorical coherence object} in $\cat$ (a special bicategorical weighted colimit), which can be constructed from bicategorical coinserters and bicategorical coequifiers. Using a ``diagonal lemma'' for reflexive coherence objects (a categorification of the ``diagonal lemma'' for reflexive coequalizers), one can endow that codescent object with the structure of a tensor category. One then has to prove that it is actually a coherence object in $\cat_{\otimes}$. Now one uses that a coherence object of the above data in $\cat_{\otimes}$ is just a bicategorical pushout of $\A \leftarrow \C \rightarrow \B$ by \cite[Proposition 4.3]{Sch18}. 
\end{rem}

\begin{rem} \label{coinvex}
It follows from \cref{tensorcoins} and \cref{tensorcoequifiers} via \cref{coeqviains} that $\cat_{\otimes}$ has bicategorical coinverters. But their construction can be done in one step, using bicategorical coinverters in $\cat$. Namely, if $\alpha : F \to G : \A \rightrightarrows \B$ is a morphism of tensor functors, its bicategorical coinverter can be constructed from the bicategorical coinverter of $\overline{\alpha} : \overline{F} \to \overline{G} : \B \times \A \rightrightarrows \B$ in $\cat$. The proof is similar to that of \cref{coequifish}.
\end{rem}

Now let us generalize our arguments to the other types of tensor categories.

\begin{prop} \label{coinscoeqtensor}
The $2$-categories $\cat_{\otimes}$, $\cat_{\otimes/\IK}$, $\cat_{\fc\!\otimes}$, $\cat_{\fc\!\otimes/\IK}$ have bicategorical coinserters and coequifiers.
\end{prop}

\begin{proof}
We just treat $\cat_{\fc\!\otimes}$, the $\IK$-linear variants are similar. First, \cref{preservecolim} generalizes to the functor $\boxtimes : \cat_{\fc}^2 \to \cat_{\fc}$ since $\A \boxtimes -$ is bicategorically left adjoint to $\HOM_{\fc}(\A,-)$.

Next, \cref{coinsish} generalizes to finitely cocontinuous functors $F,G : \I \rightrightarrows \A$ from a small finitely cocomplete category $\I$ into a small finitely cocomplete tensor category $\A$. In the proof, we define the finitely cocontinuous functor $\overline{F} : \A \boxtimes \I \to \A$ using the universal property of $\boxtimes$ by $A \boxtimes i \mapsto A \otimes F(i)$, similarly $\overline{G} : \A \boxtimes \I \to \A$. Then we define $(P : \A \to \B, \, \delta : P \circ \overline{F} \to P \circ \overline{G})$ as a bicategorical coinserter in $\cat_{\fc}$ (which exists by \cref{coins}). By the universal property of  $\boxtimes$, the morphism $\delta$ is completely determined by its behavior on elementary tensors, i.e.\ by natural morphims $P(A \otimes F(i)) \to P(A \otimes G(i))$. The finitely cocontinuous functor $\A \boxtimes \A \to \B$, $A \boxtimes B \mapsto P(A \otimes B)$ coinserts the pairs $\overline{F} \boxtimes \A,\, \overline{G} \boxtimes \A$ and $\A \boxtimes \overline{F},\, \A \boxtimes \overline{G}$ using morphisms which can be defined in the same way as in the previous proof; notice that because of the universal property of $\boxtimes$ it suffices to construct them on elementary tensors. Thus by \cref{preservecolim} there is a functor $\B \boxtimes \B \to \B$, which thus corresponds to a functor $\otimes : \B^2 \to \B$ which is finitely cocontinuous in each variable, together with natural isomorphisms $\mu : P(A) \otimes P(B) \myiso P(A \otimes B)$ which satisfy the same diagrams (\ref{diag1}) and (\ref{diag2}). The rest of the proof can now easily be copied.

Let us remark that the remarks after that proposition hold in a very similar fashion for finitely cocomplete (symmetric) (monoidal) categories.
 
Next, \cref{coequifish} holds in the same way for morphisms $\alpha,\beta : F \rightrightarrows G$ of finitely cocontinuous functors $F,G : \I \rightrightarrows \A$ from a small finitely cocomplete category $\I$ into a small finitely cocomplete tensor category $\A$. One defines auxiliary morphims $\overline{\alpha},\overline{\beta} : \overline{F} \rightrightarrows \overline{G}$ and takes their coequifier $P : \A \to \B$ (which exists by \cref{coequifiers}). The finitely cocontinuous functor $\A \boxtimes \A \to \B$, $A \boxtimes B \mapsto P(A \otimes B)$ coequifies $\overline{\alpha},\overline{\beta}$ in each variable: It suffices to check this on elementary tensors, so we may just repeat the proof. Hence, we obtain a finitely cocontinuous functor $\B \boxtimes \B \to \B$ etc.\

As in \cref{tensorcoequifiers} it follows immediately from the previous result that $\cat_{\fc\!\otimes}$ has bicategorical coequifiers.

Finally, as in the proof of \cref{tensorcoins} one can use the previous results to show the existence of bicategorical coinserters in $\cat_{\fc\!\otimes}$. In the proof we replace $\A^2$ by $\A \boxtimes \A$ and $\A^0$ by the unit for $\boxtimes$, i.e.\ $\FinSet$ (resp.\ $\Mod_{\fp}(\IK)$ for in the $\IK$-linear case).
\end{proof}

\begin{rem} \label{unify2}
We conjecture that more generally for a symmetric monoidal bicategory $(\C,\boxtimes)$ such that $\C$ has bicategorical coinserters and coequifiers which are preserved by $\boxtimes$ in each variable, its bicategory $\SymPsMon(\C,\boxtimes)$ of symmetric pseudomonoids has bicategorical coinserters and coequifiers as well. As in \cref{unify1}, this general result is much more advanced and not strictly necessary for our four types of tensor categories, where, as we have seen, more direct arguments are also available. We can also recast \cref{twosteps} here: If $\C$ just has bicategorical coinserters which are preserved by $\boxtimes$ in each variable, this is usually not enough to show the existence of bicategorical coinserters in $\SymPsMon(\C,\boxtimes)$.
\end{rem}

\begin{cor} \label{tensorcocomplete}
The $2$-categories $\cat_{\otimes}$, $\cat_{\otimes/\IK}$, $\cat_{\fc\!\otimes}$, $\cat_{\fc\!\otimes/\IK}$ have bicategorical pushouts. In fact, they are bicategorically cocomplete.
\end{cor}

\begin{proof}
The first statement follows from \cref{coprodtensor} and \cref{coinscoeqtensor} using \cref{pushviacoprodcoeqcoins}, the second from \cref{allcoprodtensor} and \cref{coinscoeqtensor} using \cref{allcolim}.
\end{proof}

\section{Locally finitely presentable categories}

We now apply the results from the previous section to the $2$-categories $\LFP$ and $\LFP_{\otimes}$ of locally finitely presentable (tensor) categories together with cocontinuous tensor functors preserving finitely presentable objects; everything here holds in a similar way for the $\IK$-linear variants $\LFP_{\IK}$ and $\LFP_{\otimes/\IK}$. We first record the special case which will be used in \cite{Bra20}.

\begin{prop} \label{pushy}
Given two morphisms $\A \leftarrow \C \rightarrow \B$ in $\LFP_{\otimes/\IK}$, their bicategorical pushout exists in $\LFP_{\otimes/\IK}$. It is even the bicategorical pushout in $\Cat_{\c\otimes/\IK}$.
\end{prop}
 
\begin{proof}
We use the same method as in the proof of \cref{coprodlfp}. The given morphisms correspond to two morphisms $\A_{\fp} \leftarrow \C_{\fp} \rightarrow \B_{\fp}$ in $\cat_{\fc\!\otimes/\IK}$, which have a bicategorical pushout $\A_{\fp} \boxtimes_{\C_{\fp}} \B_{\fp}$ by \cref{tensorcocomplete}. We claim that $\smash{\A \coboxtimes_\C \B \coloneqq \Ind(\A_{\fp} \boxtimes_{\C_{\fp}} \B_{\fp})}$ is the bicategorical pushout in the $2$-category of all cocomplete $\IK$-linear tensor categories $\Cat_{\c\otimes/\IK}$. Using \cref{smallred}, for $\D \in \Cat_{\c\otimes/\IK}$ we calculate  (where $\D'$ runs through all small full subcategories of $\D$ which are closed under finite colimits and finite tensor products)
\begin{align*}
& ~~~\,\, \smash{\Hom_{\c\otimes/\IK}\bigl(\A \coboxtimes_\C \B,\D\bigr)} \\
& \simeq \Hom_{\fc\!\otimes/\IK}(\A_{\fp} \boxtimes_{\C_{\fp}} \B_{\fp},\D) \\
& \simeq {\varinjlim}_{\D'} \Hom_{\fc\!\otimes/\IK}(\A_{\fp} \boxtimes_{\C_{\fp}} \B_{\fp},\D') \\
& \simeq {\varinjlim}_{\D'} \bigl(\Hom_{\fc\!\otimes/\IK}(\A_{\fp},\D') \, \times_{\Hom_{\fc\!\otimes/\IK}(\C_{\fp},\D')} \, \Hom_{\fc\!\otimes/\IK}(\B_{\fp},\D')\bigr) \\
& \simeq {\varinjlim}_{\D'} \Hom_{\fc\!\otimes/\IK}(\A_{\fp},\D') \, \times_{{\varinjlim}_{\D'} \Hom_{\fc\!\otimes/\IK}(\C_{\fp},\D')} \, {\varinjlim}_{\D'} \Hom_{\fc\!\otimes/\IK}(\B_{\fp},\D') \\
& \simeq \Hom_{\fc\!\otimes/\IK}(\A_{\fp},\D)\,  \times_{\Hom_{\fc\!\otimes/\IK}(\C_{\fp},\D) } \, \Hom_{\fc\!\otimes/\IK}(\B_{\fp},\D)  \\
& \simeq \Hom_{\c\otimes/\IK}(\A,\D) \, \times_{\Hom_{\c\otimes/\IK}(\C,\D) } \, \Hom_{\c\otimes/\IK}(\B,\D).
\end{align*}
This shows the desired universal property.
\end{proof}

Now let us treat the general case.

\begin{prop} \label{lfpcocomplete}
The $2$-categories $\LFP$ and $\LFP_{\otimes}$ are bicategorically cocomplete. 
\end{prop}

\begin{proof}
The $2$-functors $\Ind : \cat_{\fc} \to \LFP$ and $\Ind : \cat_{\fc\!\otimes} \to \LFP_{\otimes}$ are equivalences of $2$-categories; the pseudo-inverse $2$-functors are given by $\A \mapsto \A_{\fp}$ in each case. Since we know from \cref{catcocomplete} and \cref{tensorcocomplete} that $\cat_{\fc}$ and $\cat_{\fc\!\otimes}$ are bicategorically cocomplete, the claim follows.
\end{proof}

Next, we observe that the universal properties of bicategorical weighted colimits in $\LFP$ and $\LFP_{\otimes}$ actually hold in the larger $2$-categories $\Cat_{\c}$ and $\Cat_{\c\otimes}$ of cocomplete (tensor) categories. Again, we can use the same proof as in \cref{coprodlfp}.
 
\begin{lemma} \label{colimremain}
The inclusion $2$-functors $\LFP \hookrightarrow \Cat_{\c}$ and $\LFP_{\otimes} \hookrightarrow \Cat_{\c\otimes}$ preserve all bicategorical weighted colimits.
\end{lemma}

\begin{proof} 
Let $\J$ be a small bicategory, $X : \J^{\op} \to \cat$ be a weight and $\A : \J \to \LFP$ be a homomorphism of bicategories. By \cref{lfpcocomplete} the bicategorical weighted colimit $X *_b \A$ in $\LFP$ can be constructed as follows: Let $\A_{\fp} : \J \to \cat_{\fc}$ be defined by $\A_{\fp}(j) \coloneqq \A(j)_{\fp}$, so that $\A(j) = \Ind(\A_{\fp}(j))$. Then $X *_b \A = \Ind(X *_b \A_{\fp})$, where $X *_b \A_{\fp}$ is the bicategorical weighted colimit in $\cat_{\fc}$. For $\C \in \Cat_{\c}$ we have (where $\C'$ runs through all small full subcategories of $\C$ which are closed under finite colimits)
\begin{align*}
\Hom_{\c}(X * _b \A, \C) & \simeq \Hom_{\fc}(X *_b \A_{\fp},\C) \\
& \simeq {\varinjlim}_{\C'} \Hom_{\fc}(X *_b \A_{\fp},\C') \\
& \simeq {\varinjlim}_{\C'} \Hom\bigl(X(-),\Hom_{\fc}(\A_{\fp}(-),\C')\bigr) \\
& \simeq \Hom\bigl(X(-),{\varinjlim}_{\C'}  \Hom_{\fc}(\A_{\fp}(-),\C')\bigr) \quad (\ast) \\
& \simeq \Hom\bigl(X(-),\Hom_{\fc}(\A_{\fp}(-),\C)\bigr) \\
& \simeq \Hom\bigl(X(-),\Hom_{\c}(\A(-),\C)\bigr).
\end{align*}
In $(\ast)$ we have used that $\J$ and each $X(j)$ is small. This shows that  $X *_b \A$ has the desired universal property. The proof for $\LFP_{\otimes} \hookrightarrow \Cat_{\c\otimes}$ is very similar.
\end{proof}
 
We will now describe the bicategorical colimits in $\LFP$ and $\LFP_{\otimes}$ in a much more concrete fashion. At least we will be able to do this for the underlying categories of these bicategorical colimits, since the description of colimits (as well as tensor products in the case of $\LFP_{\otimes}$) is a much more complicated issue. Let us remark without proof that, at least for $\LFP$, these descriptions can also be obtained from \cite[Proposition 2.1.11]{CJF13} and actually also hold without the requirement that the tensor functors preserve finitely presentable objects (for this one needs the theory of accessible categories).

\begin{rem}\label{indes}
If $\C$ is a small finitely cocomplete category, then a presheaf $P : \C^{\op} \to \Set$ is a filtered colimit of representables, i.e.\ belongs to $\Ind(\C)$, if and only if $P$ is finitely continuous. For a proof, see \cite[Proposition 6.1.2]{Bor94}. Thus, we have
\[\Ind(\C) = \Hom_{\fc}(\C,\Set^{\op})^{\op}.\]
\end{rem}
 
\begin{prop} \label{lfpcoins}
Let $F,G : \A \rightrightarrows \B$ be two morphisms in $\LFP$. The bicategorical coinserter of $F,G$ is explicitly given by the category $\D$ of those $(B,\alpha)$, where $B \in \B$ and
\[\alpha : \Hom(G(-),B) \to \Hom(F(-),B)\]
is a morphism of functors $\A^{\op} \to \Set$. The bicategorical coisoinserter (coequalizer) of $F,G$ is the full subcategory of those $(B,\alpha)$ for which $\alpha$ is an isomorphism.
\end{prop}

\begin{proof}
Consider the morphisms $F_{\fp},G_{\fp} : \A_{\fp} \rightrightarrows \B_{\fp}$ in $\cat_{\fc}$ and their bicategorical coinserter $\B_{\fp} \to \C$. The bicategorical coinserter of $F,G$ is thus given by $\B = \Ind(\B_{\fp}) \to \Ind(\C)$. By \cref{indes} $\Ind(\C)$ is opposite to the category of finitely cocontinuous functors $\C \to \Set^{\op}$, which by the universal property of $\C$ (applied to a suitable small portion of $\Set^{\op}$) correspond to finitely cocontinuous functors $H : \B_{\fp} \to \Set^{\op}$ equipped with a morphism $H \circ F_{\fp} \to H \circ G_{\fp}$. These, in turn, correspond to continuous functors $K : \B^{\op} \to \Set$ equipped with a morphism $K \circ G \to K \circ F$. Since $\B$ is locally presentable, Freyd's special adjoint functor theorem implies that $K$ is representable, so we are done. For bicategorical coisoinserters a similar proof works.
\end{proof}

\begin{rem}
In the description of the bicategorical coinserter $\D$ of $F,G : \A \rightrightarrows \B$ in \cref{lfpcoins} it is unclear how colimits in $\D$ look like; also the universal functor $P : \B \to \D$ is not easy to see. In fact, there is no direct construction. Reflection techniques (or transfinite compositions) are necessary. The forgetful functor $\D \to \B$, $(B,\alpha) \mapsto B$ is actually the right adjoint of $P$, and $\alpha$ is actually the mate of the universal morphism $P \circ F \to P \circ G$. Similar remarks apply to the following examples as well.
\end{rem}

\begin{prop} \label{lfpcoequi}
Let $F,G : \A \rightrightarrows \B$ be two morphisms in $\LFP$ and $\alpha,\beta : F \rightrightarrows G$ be two morphisms. The bicategorical coequifier of $\alpha,\beta$ is explicitly given by the category of those $B \in \B$ for which $\alpha^*,\beta^* : \Hom(G(-),B) \rightrightarrows \Hom(F(-),B)$ are equal.
\end{prop}

\begin{proof}
We can use the same method of proof as in \cref{lfpcoins}. Consider the morphisms $F_{\fp},G_{\fp} : \A_{\fp} \rightrightarrows \B_{\fp}$ in $\cat_{\fc}$, the morphisms $\alpha_{\fp},\beta_{\fp} : F_{\fp} \rightrightarrows G_{\fp}$ and their bicategorical coequifier $\B_{\fp} \to \C$. Then $\B = \Ind(\B_{\fp}) \to \Ind(\C)$ is the bicategorical coequifier of $\alpha,\beta$. By \cref{indes} $\Ind(\C)$ is opposite to the category of finitely cocontinuous functors $\C \to \Set^{\op}$, which correspond to finitely cocontinuous functors $H : \B_{\fp} \to \Set^{\op}$ which coequify $\alpha_{\fp},\beta_{\fp}$, which in turn correspond to continuous functors $K : \B^{\op} \to \Set$ which coequify $\alpha^{\op},\beta^{\op}$. But then $K$ is representable, so we are done.
\end{proof}

\begin{prop} \label{lfpcoinv}
Let $F,G : \A \rightrightarrows \B$ be two morphisms in $\LFP$ and $\alpha : F \to G$ be a morphism. The bicategorical coinverter of $\alpha$ is explicitly given by the category of those $B \in \B$ for which $\alpha^* : \Hom(G(-),B) \to \Hom(F(-),B)$ is an isomorphism.
\end{prop}

\begin{proof}
We can use the same method of proof as in \cref{lfpcoins}.
\end{proof}

\begin{prop} \label{lfptensor}
Let $\A,\B$ be two objects of $\LFP$. Then their tensor product $\smash{\A \coboxtimes \B}$, which classifies functors on $\A \times \B$ which are cocontinuous in each variable, cf.\ \cite[Corollary 2.2.5]{CJF13}, is explicitly given by the category of those presheaves $\A_{\fp}^{\op} \times \B_{\fp}^{\op} \to \Set$ which are finitely continuous in each variable.
\end{prop}

\begin{proof}
By \cref{indes} the category $\smash{\A \coboxtimes \B \coloneqq \Ind(\A_{\fp} \boxtimes \B_{\fp})}$ is opposite to the category of finitely cocontinuous functors $\A_{\fp} \boxtimes \B_{\fp} \to \Set^{\op}$. By the universal property of $\boxtimes$, these correspond to functors $\A_{\fp} \times \B_{\fp} \to \Set^{\op}$ which are finitely cocontinuous in each variable.
\end{proof}

Now let us look at $\LFP_{\otimes}$. Notice that \cref{lfptensor} already gives a description of binary bicategorical coproducts in $\LFP_{\otimes}$ (at least, of their underlying categories).

\begin{prop} \label{lfpcoinstensor}
Let $F,G : \A \rightrightarrows \B$ be two morphisms in $\LFP_{\otimes}$. The bicategorical coinserter of $F,G$ is explicitly given by the category of those $(B,\alpha)$, where $B \in \B$ and
\[\alpha : \HOM(G(-),B) \to \HOM(F(-),B)\]
is a morphism of functors $\A^{\op} \to \B$ which is compatible with the tensor structure on $\A$, i.e.\ two following diagrams commute:
\[\begin{tikzcd}[row sep = 6ex]
\HOM(G(1),B) \ar{rr}{\alpha} && \HOM(F(1),B) \\ & B \ar{ul}{\sim} \ar{ur}[swap]{\sim} & 
\end{tikzcd}\]
\[\begin{tikzcd}[row sep = 6ex]
\HOM(G(A \otimes A'),B) \ar{d}[swap]{\alpha} \ar{r}{\sim} & \HOM\bigl(G(A'),\HOM(G(A),B)\bigr) \ar{r}{\alpha} & \HOM\bigl(G(A'),\HOM(F(A),B)\bigr) \ar{d}{\sim} \\
\HOM(F(A \otimes A'),B) & \HOM\bigl(F(A),\HOM(F(A'),B)\bigr)  \ar{l}{\sim} & \HOM\bigl(F(A),\HOM(G(A'),B)\bigr) \ar{l}{\alpha}
\end{tikzcd}\]
\vphantom{i am a space}

\noindent The bicategorical coisoinserter (coequalizer) of $F,G$ is the full subcategory of those $(B,\alpha)$ for which $\alpha$ is an isomorphism.
\end{prop}

\begin{proof}
The bicategorical coinserter of $F,G$ is $\Ind(\D)$, where $\B_{\fp} \to \D$ is the bicategorical coinserter of $F_{\fp},G_{\fp} : \A_{\fp} \rightrightarrows \B_{\fp}$ in $\cat_{\fc\!\otimes}$, whose construction we recall from the proof of \cref{coinscoeqtensor}. We first consider the auxiliary morphisms $\overline{F}_{\fp},\overline{G}_{\fp} : \B_{\fp} \boxtimes \A_{\fp} \to \B_{\fp}$ and take their bicategorical coinserter $(P : \B_{\fp} \to \C, \, \delta : P \circ \overline{F}_{\fp} \to P \circ \overline{G}_{\fp})$ in $\cat_{\fc}$. Then $P$ actually has a tensor structure and, as such, is universal with $\delta' : P \circ F_{\fp} \to P \circ G_{\fp}$. But $\delta'$ is not a morphism of tensor functors. To fix this, one constructs the multiple bicategorical coequifier $Q : \C \to \D$ in $\cat_{\fc}$ of two certain morphisms between two functors $\C \rightrightarrows \C$ (which ensure compatibility with the unit) and two certain morphisms between two functors $\C \boxtimes \A_{\fp} \boxtimes \A_{\fp} \rightrightarrows \C$ (which ensure compatibility with the tensor product). Thanks to our ``stabilization'' with $\C$ we can endow $Q$ with a tensor structure, and $Q \circ P : \B_{\fp} \to \D$ is the desired bicategorical coinserter of $F_{\fp},G_{\fp}$.

It follows from \cref{indes} and the universal property of $\D$ in $\cat_{\fc}$ that $\Ind(\D)^{\op}$ is equivalent to the category of those finitely cocontinuous functors $H : \C \to \Set^{\op}$ which coequify the two mentioned pairs of morphisms $(\ast)$. Because of the universal property of $\C$, $H$ corresponds to a finitely cocontinuous functor $\B_{\fp} \to \Set^{\op}$, or equivalently a continuous functor $K : \B^{\op} \to \Set$, together with a morphism of functors $K \circ \overline{G}_{\fp}^{\op} \to K \circ \overline{F}_{\fp}^{\op}$. This morphism is determined by natural maps
\[K(B' \otimes G(A)) \to K(B' \otimes F(A))\]
for $A \in \A_{\fp}$, $B' \in \B_{\fp}$. Again, $K$ is representable by Freyd's adjoint functor theorem, so that we may assume $K = \Hom(-,B)$ and rewrite the maps as
\[\Hom\bigl(B',\HOM(G(A),B)\bigr) \to \Hom\bigl(B',\HOM(F(A),B)\bigr).\]
By the Yoneda Lemma, these reduce to morphisms
\[\alpha : \HOM(G(A),B) \to \HOM(F(A),B)\]
in $\B$, natural in $A \in \A_{\fp}$. It is a straight forward exercise to check that the two conditions $(\ast)$ on $H$ translate to the two conditions on $\alpha$ in the proposition. The proof for bicategorical coisoinserters is similar.
\end{proof}

In particular, we can answer the question asked in \cite[Remark 5.1.15]{Bra14} in a special case, namely how to freely adjoin a morphism to a cocomplete tensor category between two given objects.

\begin{cor}
Let $\A \in \LFP_{\otimes}$ and $U,V \in \A_{\fp}$. Then there is a universal cocomplete tensor category $\A[U \to V]$ with a cocontinuous tensor functor $P : \A \to \A[U \to V]$ and a morphism $P(U) \to P(V)$. It is explicitly given by the category of those $(A,\alpha)$, where $A \in \A$ and
\[\alpha : \HOM(V,A) \to \HOM(U,A)\]
is a morphism which is ``compatible with itself'', which means that the following diagram (where $U',V',\alpha'$ are copies of $U,V,\alpha$) commutes:
\[\begin{tikzcd}[column sep=4ex, row sep = 4ex]
\HOM(V,\HOM(V',A)) \ar{r}{\alpha'} \ar{d}[swap]{\sim} & \HOM(V,\HOM(U',A)) \ar{r}{\sim} & \HOM(U',\HOM(V,A)) \ar{d}{\alpha}  \\
\HOM(V',\HOM(V,A)) \ar{d}[swap]{\alpha} & & \HOM(U',\HOM(U,A)) \ar{d}{\sim}\\
\HOM(V',\HOM(U,A)) \ar{r}[swap]{\sim} & \HOM(U,\HOM(V',A)) \ar{r}[swap]{\alpha'} &  \HOM(U,\HOM(U',A)) 
\end{tikzcd}\]
\end{cor}

\begin{proof}
Consider $\Set[X]$, the free cocomplete tensor category on an object $X$ \cite[Remark 5.1.14]{Bra14}. Then $\Set[X] \in \LFP_{\otimes}$ and $U,V \in \A_{\fp}$ correspond to morphisms $F,G : \Set[X] \rightrightarrows \A$ in $\LFP_{\otimes}$ via $F(X)=U$ and $G(X)=V$. Now the claim follows easily from \cref{lfpcoinstensor} and \cref{colimremain}.
\end{proof}

\begin{prop} \label{lfpcoequitensor}
Let $F,G : \A \rightrightarrows \B$ be two morphisms in $\LFP_{\otimes}$ and $\alpha,\beta : F \rightrightarrows G$ be two morphisms. The bicategorical coequifier of $\alpha,\beta$ is explicitly given by the category of those $B \in \B$ for which $\alpha^*,\beta^* : \HOM(G(-),B) \rightrightarrows \HOM(F(-),B)$ are equal.
\end{prop}

\begin{proof}
The proof is similar to that of \cref{lfpcoinstensor} and uses the construction of bicategorical coequifiers in $\cat_{\fc\!\otimes}$, cf.\ \cref{coinscoeqtensor}.
\end{proof}

\begin{prop} \label{lfpcoinvtensor}
Let $F,G : \A \rightrightarrows \B$ be two morphisms in $\LFP_{\otimes}$ and $\alpha : F \to G$ be a morphism. The bicategorical coinverter of $\alpha$ is explicitly given by the category of those $B \in \B$ for which $\alpha^* : \HOM(G(-),B) \to \HOM(F(-),B)$ is an isomorphism.
\end{prop}

\begin{proof}
The proof is similar to that of \cref{lfpcoinstensor} and uses the construction of bicategorical coinverters in $\cat_{\fc\!\otimes}$, cf.\ \cref{coinvex}.
\end{proof}

\begin{prop} \label{lfppushtensor}
Let $F : \C \to \A$ and $G : \C \to \B$ be two morphisms in $\LFP_{\otimes}$. Their bicategorical pushout $\smash{\A \coboxtimes_\C \B}$ in $\LFP_{\otimes}$ is explicitly given by the category of those $(P,\alpha)$, where $ P : \A_{\fp}^{\op} \times \B_{\fp}^{\op} \to \Set$ is a presheaf which is finitely continuous in each variable and $\alpha$ is a family of natural isomorphisms
\[\alpha_{A,B,C} : P\bigl(A,G(C) \otimes B\bigr) \myiso P\bigl(A \otimes F(C),B\bigr)\]
for $A \in \A_{\fp}$, $B \in \B_{\fp}$, $C \in \C_{\fp}$ which are compatible with the tensor structure on $\C_{\fp}$, i.e.\ the two following diagrams commute: 
\[\begin{tikzcd}[row sep = 5ex]
P\bigl(A,G(1_\C) \otimes B) \ar{rr}{\alpha} & & P(A \otimes F(1_\C),B) \\ 
 & P(A,B) \ar{ul}{\sim} \ar{ur}[swap]{\sim} & 
\end{tikzcd}\]
\[\begin{tikzcd}[row sep = 5ex, column sep = 5ex]
P\bigl(A,G(C \otimes C') \otimes B\bigr) \ar{rr}{\alpha} && P\bigl(A \otimes F(C \otimes C'),B\bigr) \\
P\bigl(A,G(C) \otimes (G(C') \otimes B)\bigr) \ar{r}[swap]{\alpha} \ar{u}{\sim} & P\bigl(A \otimes F(C),G(C') \otimes B\bigr) \ar{r}[swap]{\alpha} & P\bigl((A \otimes F(C)) \otimes F(C'),B\bigr) \ar{u}[swap]{\sim}
\end{tikzcd}\]
\end{prop}

\begin{proof}
The proof is similar to that of \cref{lfpcoinstensor} and uses the construction of bicategorical pushouts in $\cat_{\fc\!\otimes}$, cf.\ \cref{normaltensorpush}.
We have $\smash{\A \coboxtimes_\C \B = \Ind(\A_{\fp} \boxtimes_{\C_{\fp}} \B_{\fp})}$, which is opposite to the category of finitely cocontinuous functors $\A_{\fp} \boxtimes_{\C_{\fp}} \B_{\fp} \to \Set^{\op}$. Now $\A_{\fp} \boxtimes_{\C_{\fp}} \B_{\fp}$ is constructed from $\A_{\fp} \boxtimes \B_{\fp}$ using a bicategorical coequalizer of $\A_{\fp} \boxtimes \B_{\fp} \boxtimes \C_{\fp} \rightrightarrows \A_{\fp} \boxtimes \B_{\fp}$ in $\cat_{\fc}$ followed by a multiple bicategorical coinserter in $\cat_{\fc}$. The claim easily follows.
\end{proof}


\end{document}